\documentclass[12pt]{amsart}

\usepackage{amsthm}
\newtheorem{theorem}{Theorem}[section]
\newtheorem{lemma}[theorem]{Lemma}

\theoremstyle{definition}
\newtheorem{definition}[theorem]{Definition}

\theoremstyle{remark}
\newtheorem{remark}[theorem]{Remark}
\newtheorem{example}[theorem]{Example}
\newtheorem{corollary}[theorem]{Corollary}

\counterwithin*{section}{part}

\usepackage{amssymb}
\usepackage{empheq}
\usepackage{stmaryrd}
\usepackage{enumerate}
\usepackage[shortlabels]{enumitem}
\usepackage{calc}
\usepackage{url}
\usepackage{comment}
\usepackage{mathabx}

\usepackage{tikz}

\usepackage{xcolor}

\newcommand{\norm}[1]{\left\lVert#1\right\rVert}
\usepackage{mathtools}
\DeclarePairedDelimiter{\ceil}{\lceil}{\rceil}

\def\EE{\mathbb{E}}
\def\PP{\mathbb{P}}
\def\NN{\mathbb{N}}
\def\RR{\mathbb{R}}

\usepackage[left=2.0cm,%
                right=2.0cm,%
                top=2.5cm,%
                bottom=3.5cm,%
                headheight=12pt,%
                a4paper]{geometry}%

\begin{document}

\title[Asymptotic regularity of a generalised stochastic Halpern scheme]{Asymptotic regularity of a generalised stochastic Halpern scheme}

\author[N. Pischke and T. Powell]{Nicholas Pischke and Thomas Powell}
\date{\today}
\thanks{{\bf Funding:} The first author was supported by the `Deutsche Forschungs\-gemein\-schaft' Project DFG KO 1737/6-2. The second author was supported by the EPSRC grant EP/W035847/1.\\
{\bf Acknowledgments:} We want to thank the two anonymous referees for their helpful suggestions which improved the paper in various places.\\
{\bf Conflicts of interests:} The authors have no relevant financial or non-financial interests to disclose.\\
{\bf Data availability:} Data sharing not applicable to this article as no data sets were generated or analyzed.
}
\maketitle
\vspace*{-5mm}
\begin{center}
{\scriptsize 
Department of Computer Science, University of Bath,\\
Claverton Down, Bath, BA2 7AY, United Kingdom,\\
E-mails: $\{$nnp39,trjp20$\}$@bath.ac.uk}
\end{center}

\maketitle
\begin{abstract}
We provide abstract, general and highly uniform rates of asymptotic regularity for a generalized stochastic Halpern-style iteration, which incorporates a second mapping in the style of a Krasnoselskii-Mann iteration. This iteration is general in two ways: First, it incorporates stochasticity completely abstractly, rather than fixing a sampling method; second, it includes as special cases stochastic versions of various schemes from the optimization literature, including Halpern's iteration as well as a Krasnoselskii-Mann iteration with Tikhonov regularization terms in the sense of Bo\c{t}, Csetnek and Meier (where this stochastic variant of the latter is considered for the first time in this paper). For these specific cases, we obtain linear rates of asymptotic regularity, matching (or improving) the currently best known rates for these iterations in stochastic optimization, and quadratic rates of asymptotic regularity are obtained in the context of inner product spaces for the general iteration. We conclude by discussing how variance can be managed in practice through sampling methods in the style of minibatching, how our convergence rates can be adapted to provide oracle complexity bounds, and by sketching how the schemes presented here can be instantiated in the context of reinforcement learning to yield novel methods for Q-learning.
\end{abstract}

\section{Introduction}

\subsection{Background and motivation}

Approximating fixed points of nonexpansive mappings is one of the most fundamental tasks in nonlinear analysis and optimization. The problem becomes particularly interesting when we only have noisy versions of those mappings, in which case the resulting approximation methods become stochastic processes. Model-free reinforcement learning algorithms represent just one well-known example of this general situation, where variants of $Q$-learning, for instance, can be viewed as stochastic methods for computing fixed points of nonexpansive operators.

Let $(X,\norm{\cdot})$ be a separable real-valued normed space and $T,U:X\to X$ be two nonexpansive mappings on $X$, i.e.
\[
\norm{Tx-Ty}\leq \norm{x-y}\text{ and }\norm{Ux-Uy}\leq\norm{x-y}
\]
for all $x,y\in X$. In order to approximate common fixed points of two such mappings under stochastic noise constraints, we introduce in this paper the so-called stochastic Halpern-Mann iteration, given by the scheme
\[
\begin{cases}
y_n:=(1-\alpha_n)(Tx_n+\xi_n)+\alpha_nu,\\
x_{n+1}:=(1-\beta_n)(Uy_n+\delta_n)+\beta_ny_n,
\end{cases}\tag{sHM}\label{sHM}
\]
where, over some fixed probability space $(\Omega,\mathcal F,\PP)$, $x_0$ and $u$ are arbitrary $X$-valued random variables chosen as a fixed starting point and as an anchor of the iteration, respectively, $(\xi_n),(\delta_n)$ are sequences of $X$-valued random variables representing the stochastic noise, and $(\alpha_n),(\beta_n)\subseteq [0,1]$ are suitable \emph{nonstochastic} parameter sequences.

Our main scheme (\ref{sHM}) represents a stochastic analogue of the deterministic Halpern-Mann scheme for two operators (HM), schematically appearing already in \cite{NND2012}, and recently (re-)discovered in a much more generalized setting in \cite{DP2023}. This scheme integrates two of the most prominent methods for approximating fixed points for nonexpansive mappings, the Krasnoselskii-Mann method \cite{Kra1955,Man1953} and Halpern's method \cite{Hal1967} (see also \cite{Wit1992} and \cite{Xu2002}), with the intended gain of combining the beneficial features of both, in particular the strong convergence of Halpern's method even in infinite dimensional spaces. We propose that this combination gains further significance in the stochastic setting, where we anticipate that it can be used in particular to devise novel reinforcement learning algorithms or stochastic splitting methods.

In contrast to the deterministic scheme, \eqref{sHM} is designed to capture situations in which one does not have direct access to $Tx_n$ (respectively $Uy_n$), but can only use noisy versions $\tilde T$ of $T$ (and $\tilde U$ of $U$). Intuitively, $\xi_n$ would then represent the difference between $Tx_n$ and the corresponding approximation of $Tx_n$ obtained from $\tilde T$ through a suitable sampling method (and similarly for $\delta_n$ and $Uy_n$), though our presentation is fully abstract and we will make no assumptions about $\xi_n$ or $\delta_n$, other than imposing controls on $\EE[\norm{\xi_n}]$ and $\EE[\norm{\delta_n}]$, in a way which can be easily achieved in concrete scenarios through sampling methods such as minibatching, as will also be discussed later.

To date, only the very simple instance of \eqref{sHM} corresponding to a stochastic variant of Halpern's method (H) has been previously studied. Concretely, setting $U:=\mathrm{Id}$ and $\delta_n:=0$ gives rise to the iteration
\[
x_{n+1}=(1-\alpha_n)(Tx_n+\xi_n)+\alpha_nu,\tag{sH}\label{sH}
\]
considered on an abstract level in \cite{BC2024} (for finite dimensional normed spaces), a scheme which in the Euclidean setting has recently received a great deal of attention in the context of stochastic monotone inclusion problems \cite{Caietal2022,Dia2020,LTZ2024,YR2021}. In all cases, controlling the variance of the noise terms is crucial for convergence, and this is just one of several elements that makes the analysis of stochastic schemes such as \eqref{sH} markedly different from that of their nonstochastic counterparts, some of the others being a focus on \emph{oracle} complexity, and the relevance of stochastic methods to statistics and machine learning.

Given this increasing interest in stochastic variants of classic methods, the purpose of the present paper is to broaden their current scope and provide a collection of generalised convergence results in which all of the aforementioned features (variance reduction, oracle complexity, applications in machine learning) are presented in the abstract. 

While we focus on the special case of the stochastic Halpern iteration (\ref{sH}) at several points in the paper, where it forms a useful example, our method \eqref{sHM} is certainly not limited to this special case, and in line with \cite{DP2023} encompasses stochastic variants of other well-known deterministic methods. An important example of this is represented by setting $T:=\mathrm{Id}$ and $\xi_n:=0$, as well as $u:=0$ and $\gamma_n:=1-\alpha_n$, whereby we obtain a version of the Krasnoselskii-Mann iteration with Tikhonov regularization terms (KM-T) considered in \cite{BCM2019,YZL2009} that now incorporates stochastic noise, taking the form of
\[
x_{n+1}=(1-\beta_n)(U(\gamma_n x_n)+\delta_n)+\beta_n(\gamma_n x_n).\tag{sKM-T}\label{sKMT}
\]
Alternatively, the method can be seen as the stochastic Krasnoselskii-Mann iteration as considered in \cite{BravCom2024} with Tikhonov regularization terms as considered in \cite{BCM2019}. This method is known to produce fast asymptotic behavior in the deterministic setting, and in contrast to the standard Krasnoselskii-Mann scheme, also benefits from strong convergence results similarly to Halpern's iteration \cite{BCM2019}. As we will show in this paper, in particular the first of these features extends to the stochastic setting (while we also lay the foundation for the second, as discussed later). For a simple overview, we present the relationships between the different stochastic schemes considered here, as well as their deterministic counterparts, in Figure \ref{figure:flow}.

\begin{figure}[h]
\centering
\begin{tikzpicture}
\node (sHM) at (0, 0) {sHM};
\node (sH) at (3, 2) {sH};
\node (sKM-T) at (3.4, -2) {sKM-T};
\node (HM) at (5, 0) {HM};
\node (H) at (8, 2) {H};
\node (KM-T) at (8.4, -2) {KM-T};
\draw [->] (sHM) -- (sH) node[midway,above,sloped] () {\tiny{$U=\mathrm{Id}, \delta_n=0$}};
\draw [->] (sHM) -- (sKM-T) node[midway,below,sloped] () {\tiny{$T=\mathrm{Id}, \xi_n=0$}};
\draw [->] (HM) -- (H) node[midway,below,sloped] () {\tiny{$U=\mathrm{Id}$}};
\draw [->] (HM) -- (KM-T) node[midway,below,sloped] () {\tiny{$T=\mathrm{Id}$}};
\draw [dashed, ->] (sHM) -- (HM) node[midway,below,sloped] () {\tiny{$\xi_n=0, \delta_n=0$}};
\draw [dashed, ->] (sH) -- (H) node[midway,below,sloped] () {\tiny{$\xi_n=0$}};
\draw [dashed, ->] (sKM-T) -- (KM-T) node[midway,below,sloped] () {\tiny{$\delta_n=0$}};
\end{tikzpicture}
\caption{Relationships of the stochastic Halpern-type schemes and their deterministic variants}
\label{figure:flow}
\end{figure}

To our knowledge, each of the stochastic schemes \eqref{sHM} and \eqref{sKMT} are introduced here for the first time.

\subsection{Overview of the main results}

Our main results comprise both abstract convergence theorems for the stochastic schemes in Figure 1, valid under very broad assumptions (and with convergence rates given at a corresponding level of generality), along with a series of special cases where fast, linear rates are possible, the latter demonstrating that our framework includes stochastic algorithms that come with state-of-the-art convergence guarantees of a kind only recently established for deterministic Halpern-type schemes \cite{Lie2021,SS2017}. We now outline each of these in turn.

\subsubsection{General (quantitative) asymptotic regularity results}

The main results of the paper establish general conditions under which one can guarantee the asymptotic regularity of the scheme (\ref{sHM}), both in the traditional sense of $\norm{x_n-x_{n+1}}$ (sometimes called the discrete velocity \cite{BSV2023}) and also relative to the mappings, i.e.\ considering the displacements $\norm{x_n-Tx_n}$ and $\norm{x_n-Ux_n}$. Furthermore, we establish these asymptotic regularity results both in expectation and almost surely, that is we show both
\[
\EE[\norm{x_n-x_{n+1}}],\ \EE[\norm{x_n-Tx_n}],\ \EE[\norm{x_n-Ux_n}]\to 0
\]
as well as
\[
\norm{x_n-x_{n+1}},\ \norm{x_n-Tx_n},\ \norm{x_n-Ux_n}\to 0\text{ almost surely}.
\]
Most importantly, in all cases we provide explicit convergence rates for these expressions. In the case of convergence in mean, these rates take the form of functions $\varphi$ which guarantee
\[
\EE[\norm{x_n-x_{n+1}}]<\varepsilon\text{ for all }\varepsilon>0\text{ and any }n\geq\varphi(\varepsilon),
\]
and similarly for $\EE[\norm{x_n-Tx_n}]$ and $\EE[\norm{x_n-Ux_n}]$. In the case of our almost-sure convergence results, our rates instead take the form of functions $\Phi$ which guarantee
\[
\PP(\exists n\geq\Phi(\lambda,\varepsilon)(\norm{x_n-x_{n+1}}\geq\varepsilon))<\lambda
\]
for all $\varepsilon,\lambda>0$ (and similarly for $\norm{x_n-Tx_n}$ and $\norm{x_n-Ux_n}$). In all cases, these rates are explicitly constructed under very general conditions and depend only on a few moduli witnessing quantitative aspects of our main assumptions. 

These general asymptotic regularity results and their corresponding rates for the scheme \eqref{sHM}, both in the sense of the discrete velocity and relative to the mappings, can be found in Theorems \ref{res:x:ar} -- \ref{res:other:ar:as}. While the present paper features various fine-grained discussions on the precise assumptions necessary for each of these results, they essentially amount to 
\begin{itemize}
	\item the existence of a common fixed point of $T$ and $U$, \smallskip
	\item standard conditions on the parameters $(\alpha_{n})$ and $(\beta_n)$, and \smallskip
	\item variance control in the form of $\sum_{n=0}^\infty\EE[\norm{\xi_n}]<+\infty$ and $\sum_{n=0}^\infty\EE[\norm{\delta_n}]<+\infty$.
\end{itemize}

In their quantitative variants, the respective rates correspondingly rely on bounds and rates witnessing these assumptions in various ways. While the results on the discrete velocity as given in Theorems \ref{res:x:ar} and \ref{res:y:ar} are proved unconditionally, asymptotic regularity relative to the mappings turns out to be more involved, and the respective Theorems \ref{res:other:ar} and \ref{res:other:ar:as} depend a priori on the assumption that $\norm{Uy_n-y_n}\to 0$, either in mean or almost surely. It is precisely at this point that our treatment divides according to the main schemes we consider in Figure 1. For \eqref{sH} and \eqref{sKMT}, this premise either trivializes or can be directly derived without additional conditions (see Theorem \ref{thm:UAsRegKMT}), but for the general case of \eqref{sHM}, establishing $\norm{Uy_n-y_n}\to 0$ almost surely requires a subtle pointwise geometric argument based on the additional assumption that the underlying space $X$ is uniformly convex (see Theorem \ref{geometry:as}). Our corresponding result in mean (see Theorem \ref{geometry:alt}) is effectively obtained as a direct lift of the pointwise argument, which we achieve in an abstract way by showing that the sequence $(\norm{Uy_n-y_n})$ is uniformly integrable given the existence of a common fixed point and the variance assumptions $\sum_{n=0}^\infty\EE[\norm{\xi_n}]<+\infty$ and $\sum_{n=0}^\infty\EE[\norm{\delta_n}]<+\infty$. 

Naturally, these geometric and probabilistic considerations also influence the corresponding rates, where the crucial assumption of uniform convexity and passage through uniform integrability result in a dependence of the rate on a modulus measuring the degree of uniform convexity of $X$, along with a similar modulus measuring the degree of uniform integrability of $(\norm{Uy_n-y_n})$. In particular, we show how such moduli can be explicitly constructed for particular spaces, and from natural quantitative integrability assumptions on the error terms, respectively.

To the best of our knowledge, all the respective rates for \eqref{sKMT} and \eqref{sHM}, as well as even the ``qualitative'' asymptotic regularity results, that is convergence alone even without considering the quantitative aspects, are completely novel. In terms of complexity and structure, they seem to match previously constructed general rates for their deterministic analogs (KM-T) and (HM) derived in \cite{CKL2023,CL2022,CL2024,DP2021} and \cite{DP2023,LP2024}, respectively, with the additional component of a modulus of uniform integrability featuring in the case of \eqref{sHM} as mentioned before. In the case of the stochastic Halpern iteration \eqref{sH}, the only scheme already to have been studied, our convergence theorems are the best yet in that they both generalize and improve the asymptotic regularity results recently presented in \cite{BC2024}.

\subsubsection{Fast rates of asymptotic regularity}

Beyond just presenting general and rather abstract convergence results and rates as outlined above, we also identify natural circumstances under which these rates are very fast, reaching up to linear speed in special cases. For the parameter choices $\alpha_n=2/(n+2)$ and $\beta_n=\beta\in (0,1)$, together with sufficiently fast decaying errors with $\EE[\norm{\xi_n}],\EE[\norm{\delta_n}]\leq K/(n+2)^2$ for some constant $K$, we derive (see Theorem \ref{thm:fastX}) linear rates for the discrete velocity of the general scheme \eqref{sHM} in the form of
\[
\EE[\norm{x_n-x_{n+1}}]\leq \frac{K'}{n+2}\text{ and }\PP\left(\exists i\geq n\left( \norm{x_i-x_{i+1}}\geq\varepsilon\right)\right)\leq \frac{1}{\varepsilon}\frac{K'}{n+2}
\]
for all $n\in\mathbb{N}$ and $\varepsilon>0$, and some suitable (explicitly definable) constant $K'$. Similar results also hold for the discrete velocity of $(y_n)$ in mean and almost surely (see Theorem \ref{thm:fastY}). In particular, we want to note that the above results, as well as all other linear rates, are actually given in the form of nonasymptotic guarantees, also in the case of the almost sure rates.

In the case of asymptotic regularity relative to the mappings, our results again become sensitive to the choice of the scheme: The special cases of \eqref{sH} and \eqref{sKMT} exhibit linear nonasymptotic guarantees under the same assumptions on the parameters as detailed above. In the case of \eqref{sH} (see Theorem \ref{thm:fastHalpern}), we in particular have
\[
\EE[\norm{Tx_n-x_n}]\leq \frac{K'}{n+2}\text{ and }\PP\left(\exists i\geq n\left( \norm{Tx_i-x_i}\geq\varepsilon\right)\right)\leq\frac{1}{\varepsilon}\frac{K'}{n+2}
\]
for all $n\in\mathbb{N}$ and $\varepsilon>0$ and some suitable (explicitly definable) constant $K'$ similar to before. Effectively the same results hold for $U$ in place of $T$ in the case of \eqref{sKMT} (see Theorem \ref{thm:fastTM}).

The results for the general scheme \eqref{sHM} again rely both on assumptions on the geometry of the space $X$ and the uniform integrability of the sequence $(\norm{Uy_n-y_n})$. In the special case of a uniformly convex space of power type $p\geq 2$, our results in particular yield rates of order $p$ so that we obtain quadratic rates in the case of inner product spaces, with order
\[
\mathcal{O}(\varepsilon^{-2}\lambda^{-2})\text{ or } \mathcal{O}(\varepsilon^{-2}\mu(\varepsilon/4)^{-2})
\]
for $\norm{Ux_n-x_n},\norm{Tx_n-x_n}\to 0$ almost surely or in mean, respectively, where $\mu$ is our modulus of uniform integrability for $(\norm{Uy_n-y_n})$ (see Theorem \ref{thm:fastGeneral} and Remark \ref{rem:fastRateGen}).

As with our generally constructed rates, our fast rates for the schemes \eqref{sHM} and \eqref{sKMT} are to the best of our knowledge also novel. In regards to complexity, and seem to match the complexities observed in the deterministic case (we again refer to \cite{DP2023,LP2024} and \cite{CKL2023,CL2022,CL2024,DP2021}, respectively). For the scheme \eqref{sH}, our results improve the fast rates presented in \cite{BC2024} by removing the respective logarithmic factors contained therein. In particular, by restricting our attention to the deterministic Halpern iteration (H), we reobtain the linear rates of asymptotic regularity previously derived in \cite{Lie2021,SS2017} (albeit with different constants), which are known to be tight (see \cite{Lie2021}). Indeed, our linear rates are derived by adapting the approach of \cite{SS2017}). Based on their inherent similarity, it can moreover be believed that the linear rates of asymptotic regularity observed for the deterministic Krasnoselskii-Mann iteration with Tikhonov regularization terms (KM-T) are also tight, in which case our corresponding results for the stochastic scheme \eqref{sKMT} would be tight as well, although we are not aware of any results that prove this explicitly. The question for tight rates for the general Halpern-Mann scheme, both in its deterministic variant (HM) as well as in the stochastic version \eqref{sHM} presented here, remains an open problem.

\subsection{Future applications and developments}

Beyond our series of general theoretical results, one of the core motivations for this paper is the real applicability of those results in concrete areas, and we conclude with a section on how this might be achieved. First, we discuss how the requisite variance control can be managed in a practical way through the use of sampling techniques such as minibatching, and connected with this we show how our complexity results can be lifted to corresponding results on oracle complexity. We then outline several concrete applications of our methods, sketching the particularly interesting case of reinforcement learning, where the novel schemes \eqref{sHM} and \eqref{sKMT} can be instantiated in the style of Q-learning, as done recently for the stochastic Halpern iteration in \cite{BC2024}. A proper account of these applications will be provided in a forthcoming paper by the authors. Another crucial property of Halpern-type iterations not addressed in the present paper is the fact that they remain strongly convergent also in infinite dimensional spaces. These strong convergence results extend to the stochastic setting, and will be similarly addressed in forthcoming work which in particular relies on the asymptotic regularity results established here.

\subsection{This paper in connection with the proof mining program}

All of the results obtained in this paper are motivated via the methodology of the proof mining program, a subfield of mathematical logic which combines an abstract approach to proofs in mainstream mathematics with the extraction of computational information, such as bounds or rates, from those proofs. We refer to the seminal monograph \cite{Koh2008} for a comprehensive overview of both theoretical as well as applied aspects of this program, along with the survey \cite{Koh2019} for an overview of more recent applications to nonlinear analysis. Proof mining has been widely applied in nonlinear analysis, and has found particular success in providing quantitative convergence results for Halpern's iteration and its many variants, with notable instances ranging from initial rates of asymptotic regularity for Halpern's iteration given by Leu\c{s}tean \cite{Leu2007b} and the first analysis of Wittmann's proof of the strong convergence of Halpern's iteration given by Kohlenbach \cite{Koh2011}, to the extensions of these results to nonlinear context such as $\mathrm{CAT}(0)$-spaces as in \cite{KL2012} (by a logical analysis of a corresponding convergence proof by Saejung \cite{Sae2010}). They also include extensions of the Halpern iteration \cite{SK2012} for the modified Mann iteration introduced in \cite{KX2005} (and extended to nonlinear spaces in \cite{CP2011}) as well as the Krasnoselskii-Mann iteration with Tikhonov regularization terms and its extensions as in \cite{CKL2023,CL2022,CL2024,DP2021} (with \cite{CKL2023} of particular note, as linear rates of asymptotic regularity are there obtained for the first time in the context of applications of proof mining). In particular, the definition of the deterministic Halpern-Mann method given in \cite{DP2023} and its corresponding convergence proof were motivated by these logical considerations, as were the recent rates of asymptotic regularity given for this iteration in \cite{LP2024}.

The present work departs from the aforementioned case studies in nonlinear analysis in that it incorporates, for the first time, stochasticity. In this way it forms part of a recent advance of proof mining into probability theory, which comprises both new developments in the logical foundations of probability theory due to first author and Neri \cite{NPis2024}, together with applied results on the quantitative aspects of stochastic processes by the authors and Neri \cite{NPP2025,NPow2024b,NPow2024}. In particular, the present paper is one of the first applications of proof mining to stochastic optimization, and the very first to consider a concrete stochastic algorithm. It represents a particularly interesting case study in this respect, in that it does not readily follow from analogous quantitative results in the deterministic setting (such as in \cite{DP2023,LP2024}), but requires a substantial arsenal of new quantitative ideas for this stochastic setting. These include quantitative, stochastic variants of a crucial abstract lemma on recurrence inequalities due to Xu \cite{Xu2002} (Lemma \ref{stochasticbabyxu}), including an adaptation of the ``fast'' variant due to Sabach and Shtern \cite{SS2017} (Lemma \ref{sabach:stern:as}). Furthermore, proof mining also motivated the definition of a modulus of uniform integrability and its use as a suitable assumption to derive rates of asymptotic regularity in expectation in the general case of the iteration \eqref{sHM}. We envisage that all of these tools will be relevant in subsequent applications of proof mining in stochastic optimization, just as their deterministic counterparts have been used repeatedly for proof mining in nonstochastic optimization. 

We stress that while this logical perspective was crucial in obtaining the present results, the paper does not rely on any notions from logic at all.

\section{Preliminaries and basic lemmas}\label{sec:prelim}

We write $\mathbb{N}^*$ for $\mathbb{N}$ without $0$. Throughout, if not stated otherwise, we fix an underlying probability space $(\Omega,\mathcal{F},\PP)$ and all probabilistic notions such as almost sureness refer to that space. Similarly, $X$ will always denote, unless stated otherwise, a normed space with norm $\norm{\cdot}$. We refer to measurable functions $\Omega\to\mathbb{R}$ as random variables, to measurable functions $\Omega\to X$ as $X$-valued random variables and we refer to sequences of random variables as stochastic processes. In order to ensure that basic properties enjoyed by real-valued random variables are also inherited by $X$-valued random variables, so that in particular our main scheme \eqref{sHM} is well-defined, one normally requires some further assumptions on the underlying space (as discussed in detail in \cite{ProbInBanach1991}). The simplest option is to assume that $X$ is a separable Banach space, though if the reader prefers they can also just assume that $X$ is finite dimensional. Equalities and inequalities involving random variables will always be understood to hold almost surely, even if not explicitly indicated.

Throughout the paper, we will be concerned with quantitative variants of various notions and we here now briefly the discuss the key definitions of the main quantitative notions used in the paper:

Given a non-negative sequence of reals $(a_n)$, a rate of convergence for $a_n\to 0$ is a function $\varphi:(0,\infty)\to\mathbb{N}$ such that
\[
\forall\varepsilon>0\,\forall n\geq\varphi(\varepsilon)\left( a_n<\varepsilon\right).
\]
The immediate benefit of such a type of rate $\varphi$ is that if it is invertible and decreasing, then we can even derive the non-asymptotic estimate $a_n<\varphi^{-1}(n)$ for all $n\in\mathbb{N}$, which of course further implies a complexity bound on the sequence in terms of the commonly used big O notation, namely $(a_n)=O(\varphi^{-1}(n))$.

Now, given a nonnegative stochastic process $(X_n)$, a rate of convergence for $X_n\to 0$ almost surely is a function $\Phi:(0,\infty)^2\to\mathbb{N}$ such that
\[
\forall \lambda,\varepsilon>0\left( \PP\left(\exists n\geq \Phi(\lambda,\varepsilon)\, (X_n\geq \varepsilon)\right)<\lambda\right).
\]
We note that whenever $\Phi$ is a rate of convergence for $X_n\to 0$ almost surely, then for every $\varepsilon>0$, $\Phi(\varepsilon,\cdot)$ is a rate of convergence for $\PP\left(\sup_{n\geq N}\, (X_n\geq \varepsilon)\right)\to 0$ as $N\to\infty$ in the nonstochastic sense.

Further, given a non-negative sequence of reals $(a_n)$, we later want to quantitatively witness the convergence or divergence of the series over that sequence. For that, if $\sum_{n=0}^\infty a_n<\infty$, we say that a function $\chi:(0,\infty)\to\mathbb{N}$ is a rate of convergence for that sum if
\[
\forall \varepsilon>0\left( \sum_{n=\chi(\varepsilon)}^\infty a_n<\varepsilon\right).
\]
If $\sum_{n=0}^\infty a_n=\infty$, we say that a function $\theta:\mathbb{N}\times (0,\infty)\to \mathbb{N}$ is a rate of divergence for that sum if
\[
\forall b>0\,\forall k\in\mathbb{N}\left( \sum_{n=k}^{\theta(k,b)}a_n\geq b\right).
\]
Naturally, any such modulus satisfies $\theta(k,b)\geq k$ for any $k\in\mathbb{N}$ and $b>0$.

We now collect some of the basic abstract convergence results that our paper relies on. The most crucial of these, on the asymptotic behavior of sequences of reals satisfying certain recursive inequalities is the following due to Xu \cite{Xu2002}, often called Xu's lemma:

\begin{lemma}[\cite{Xu2002}]
Suppose that $(s_n),(c_n)\subseteq [0,\infty)$ as well as $(a_n)\subseteq [0,1]$ and $(b_n)\subseteq\mathbb{R}$ satisfy
\[
s_{n+1}\leq (1-a_n)s_n+a_nb_n+c_n
\]
for all $n\in\NN$ where $\sum_{n=0}^\infty a_n=\infty$, $\limsup b_n\leq 0$ and $\sum_{n=0}^\infty c_n<\infty$. Then $s_n\to 0$.
\end{lemma}

We will in particular rely on a quantitative rendering of an instance of Xu's lemma which is represented by the following lemma. This result is contained in \cite{KL2012,LP2021} (up to the way the errors and the moduli are phrased) therefore for brevity we omit the proof.

\begin{lemma}[essentially \cite{KL2012,LP2021}]\label{lem:quantXuLem}
Suppose that $(s_n),(c_n)\subseteq [0,\infty)$ and $(a_n)\subseteq [0,1]$ satisfy
\[
s_{n+1}\leq (1-a_n)s_n+c_n
\]
for all $n\in\NN$, and furthermore, that $K>0$ is an upper bound on $(s_n)$, $\theta$ is a rate of divergence for $\sum_{n=0}^\infty a_n=\infty$ and $\chi$ a rate of convergence for $\sum_{n=0}^\infty c_n<\infty$. Then $s_n\to 0$ with rate
\[
\varphi_{K,\theta,\chi}(\varepsilon):=\theta\left(\chi\left(\frac{\varepsilon}{2}\right),\ln \left(\frac{2K}{\varepsilon}\right)\right)+1.
\]
\end{lemma}

We now extend this lemma to a probabilistic variant. For that, we first consider the following result which allows us to transfer quantitative information from convergence in mean for ``almost-monotone'' sequences of random variables to rates of almost sure convergence.

\begin{lemma}\label{prob:to:as}
Let $(X_n)$, $(C_n)$ be nonnegative stochastic processes satisfying
\[
X_{n+1}\leq X_n+C_n
\]
almost surely for all $n\in\mathbb{N}$ and suppose furthermore that
\begin{enumerate}[(a)]
\item $\sum_{i=0}^\infty \EE[C_i]<\infty$ with rate $\chi$,
\item $\EE[X_n]\to 0$ with rate $\varphi$.
\end{enumerate}
Then $X_n\to 0$ almost surely, and with rate
\[
\psi(\lambda,\varepsilon):=\max\left\{\varphi(\lambda\varepsilon/2),\chi(\lambda\varepsilon/2)\right\}.
\]
\end{lemma}

\begin{proof}
We first note that for any $n\in\NN$, we have $\EE\left[\sum_{i=n}^\infty C_i\right]=\sum_{i=n}^\infty \EE[C_i]$ by the monotone convergence theorem. Now define a stochastic process $(U_n)$ by $U_n:=X_n+\sum_{i=n}^\infty C_i$. Then we have
\[
U_{n+1}=X_{n+1}+\sum_{i=n+1}^\infty C_i\leq X_n+C_n+\sum_{i=n+1}^\infty C_i\leq X_n+\sum_{i=n}^\infty C_i=U_n
\]
almost surely for any $n\in\mathbb{N}$, and therefore the events $(U_n\geq \varepsilon)$ are monotone decreasing in $n$. In particular, using Markov's inequality, we get for any $N\in\NN$:
\begin{align*}
\PP\left(\exists n\geq N(U_n\geq \varepsilon)\right)=\PP\left(U_N\geq \varepsilon\right)&\leq \frac{\EE[U_N]}{\varepsilon}=\frac{\EE[X_N]+\sum_{i=N}^\infty \EE\left[C_i\right]}{\varepsilon}.
\end{align*}
Therefore if $N=\psi(\lambda,\varepsilon)$, we have
\[
\PP\left(\exists n\geq N(U_n\geq \varepsilon)\right)\leq \frac{\EE[X_N]+\sum_{i=N}^\infty \EE\left[C_i\right]}{\varepsilon}<\frac{(\lambda\varepsilon/2+\lambda\varepsilon/2)}{\varepsilon}=\lambda.
\]
The result follows by observing that $X_n\leq U_n$ holds almost surely for all $n\in\mathbb{N}$ and thus 
\[
\PP(\exists n\geq N\, (X_n\geq\varepsilon))\leq \PP(\exists n\geq N\, (U_n\geq\varepsilon))<\lambda
\]
for any $N$, and in particular for the $N$ defined above.
\end{proof}

The above lemma now allows us to give a stochastic version of Lemma \ref{lem:quantXuLem}:

\begin{lemma}\label{stochasticbabyxu}
Suppose that $(X_n)$, $(C_n)$ are nonnegative stochastic processes satisfying
\[
X_{n+1}\leq (1-a_n)X_n+C_n
\]
almost surely for all $n\in\NN$. Furthermore, suppose that
\begin{enumerate}[(a)]
\item $\EE[X_n]\leq K$ for all $n\in\NN$,
\item $\sum_{i=0}^\infty a_i=\infty$ with rate of divergence $\theta$,
\item $\sum_{i=0}^\infty \EE[C_i]<\infty$ with rate of convergence $\chi$.
\end{enumerate}
Then $\EE[X_n]\to 0$ with rate $\varphi_{K,\theta,\chi}$ defined as in Lemma \ref{lem:quantXuLem}, i.e.
\[
\varphi_{K,\theta,\chi}(\varepsilon):=\theta\left(\chi\left(\frac{\varepsilon}{2}\right),\ln \left(\frac{2K}{\varepsilon}\right)\right)+1,
\]
and further $X_n\to 0$ almost surely with rate
\[
\psi_{K,\theta,\chi}(\lambda,\varepsilon):=\varphi_{K,\theta,\chi}\left(\frac{\lambda\varepsilon}{2}\right).
\]
\end{lemma}
\begin{proof}
Taking expectations on both sides we have
\[
\EE[X_{n+1}]\leq (1-a_n)\EE[X_n]+\EE[C_n]
\]
for all $n\in\mathbb{N}$ and so the rate for $\EE[X_n]\to 0$ follows by Lemma \ref{lem:quantXuLem}. For the rate for the almost sure convergence, observe that $\chi(\lambda\varepsilon/4)\leq \varphi_{K,\theta,\chi}(\lambda\varepsilon/2)$ as $\theta(k,b)\geq k$. Hence, one can proceed as in the proof of Lemma \ref{prob:to:as} to show that 
\[
\PP\left(\exists n\geq N(U_n\geq \varepsilon)\right)\leq \frac{\EE[X_N]+\sum_{i=N}^\infty \EE\left[C_i\right]}{\varepsilon}
\]
for any $N$ where, using $\chi(\lambda\varepsilon/4)\leq \varphi_{K,\theta,\chi}(\lambda\varepsilon/2)$, we then can conclude $\PP\left(\exists n\geq N(U_n\geq \varepsilon)\right)<\lambda$ for $N=\varphi_{K,\theta,\chi}(\lambda\varepsilon/2)$.
\end{proof}

\section{Quantitative asymptotic regularity for the generalized stochastic Halpern scheme}\label{sec:quant}

In this section we now outline our main theoretical results and derive rates of asymptotic regularity for the iterations generated by the generalized stochastic Halpern scheme.

\subsection{Basic results and rates of asymptotic regularity}

We begin with some fundamental recursive inequalities for the iterations generated by the iteration \eqref{sHM}:

\begin{lemma}[essentially \cite{LP2024}]\label{res:recurrence:y-to-x-to-y}
Let $(x_n),(y_n)$ be the sequences generated by \eqref{sHM}. Then the following recurrence relations hold pointwise everywhere for all $n\in\mathbb{N}$:
\begin{align}
\label{eqn:y-to-x}
\norm{y_{n+1}-y_n}&\leq (1-\alpha_{n+1})\left(\norm{x_{n+1}-x_n}+\norm{\xi_{n+1}-\xi_n}\right)+|\alpha_{n+1}-\alpha_n|\cdot \norm{Tx_n+\xi_n-u},\\
\label{eqn:x-to-y}
\norm{x_{n+2}-x_{n+1}}&\leq \norm{y_{n+1}-y_n}+(1-\beta_{n+1})\norm{\delta_{n+1}-\delta_n}+|\beta_{n+1}-\beta_n|\cdot \norm{Uy_n+\delta_n-y_n}.
\end{align}
\end{lemma}
\begin{proof}
For (\ref{eqn:y-to-x}) we observe that
\begin{equation*}
\begin{aligned}
\norm{y_{n+1}-y_n}=&\norm{(1-\alpha_{n+1})(Tx_{n+1}+\xi_{n+1})-(1-\alpha_{n})(Tx_{n}+\xi_{n})+(\alpha_{n+1}-\alpha_n)u}\\
\leq &(1-\alpha_{n+1})\norm{(Tx_{n+1}+\xi_{n+1})-(Tx_n+\xi_n)}\\
&+\norm{(\alpha_n-\alpha_{n+1})(Tx_n+\xi_n)-(\alpha_n-\alpha_{n+1})u}\\
\leq &(1-\alpha_{n+1})\left(\norm{x_{n+1}-x_n}+\norm{\xi_{n+1}-\xi_n}\right)+|\alpha_{n+1}-\alpha_n|\cdot \norm{Tx_n+\xi_n-u}
\end{aligned}
\end{equation*}
where for the last inequality we use that $T$ is nonexpansive. Similarly for (\ref{eqn:x-to-y}) we have
\begin{equation*}
\begin{aligned}
\norm{x_{n+2}-x_{n+1}}=&\norm{(1-\beta_{n+1})(Uy_{n+1}+\delta_{n+1})+\beta_{n+1}y_{n+1}-(1-\beta_n)(Uy_n+\delta_n)-\beta_ny_n}\\
\leq & \norm{(1-\beta_{n+1})(Uy_{n+1}+\delta_{n+1})-(1-\beta_{n+1})(Uy_n+\delta_n)+\beta_{n+1}(y_{n+1}-y_n)}\\
&+\norm{(1-\beta_{n+1})(Uy_{n}+\delta_{n})-(1-\beta_n)(Uy_n+\delta_n)-(\beta_n-\beta_{n+1})y_n}\\
\leq &(1-\beta_{n+1})\left(\norm{Uy_{n+1}-Uy_n}+\norm{\delta_{n+1}-\delta_n}\right)+\beta_{n+1}\norm{y_{n+1}-y_n}\\
&+|\beta_{n+1}-\beta_n|\cdot\norm{Uy_n+\delta_n-y_n}\\
\leq &\norm{y_{n+1}-y_n}+(1-\beta_{n+1})\norm{\delta_{n+1}-\delta_n}+|\beta_{n+1}-\beta_n|\cdot \norm{Uy_n+\delta_n-y_n}
\end{aligned}
\end{equation*}
where again we use nonexpansivity of the operator in the last step.
\end{proof}

We now move to our first quantitative result which presents a rate of asymptotic regularity for the sequence $(x_n)$, both in expectation and in probability. For that we introduce a first central assumption on the boundedness of the iteration \eqref{sHM} in expectation, as commonly made in the literature (see e.g.\ hypothesis $(\mathrm{H}_1)$ in \cite{BC2024} of which this assumption here is a natural extension to the generalised iteration \eqref{sHM}):
\begin{gather*}
\text{There exists a $K_0\in\mathbb{N}^*$ such that for all $n\in\mathbb{N}$:}\tag{Hyp}\label{Hyp}\\
\text{$\EE[\norm{Tx_n-u}],\ \EE[\norm{Uy_n-y_n}],\ \EE[\norm{Uu-u}],\ \EE[\norm{Uy_{n}-u}]\leq K_0<\infty$.}
\end{gather*}
Throughout, if not stated otherwise, we will assume the existence of such a $K_0$.

In the context of the asymptotic regularity results that hold almost surely, we will sometimes need to make a slightly stronger assumption that the random variables are actually $L^1$-bounded in the following sense:
\begin{gather*}
\text{There exists a nonnegative random variable $Y$ with $K_0\geq\EE[Y]$ for some $K_0\in\mathbb{N}^*$}\tag{Hyp$'$}\label{asHyp}\\
\text{and for all $n\in\mathbb{N}$: $\norm{Tx_n-u}, \norm{Uy_n-y_n}, \norm{Uu-u}, \norm{Uy_{n}-u}\leq Y$ almost surely.}
\end{gather*}
Contrary to the above \eqref{Hyp}, which will essentially always be tacitly assumed, we will always be very explicit about when we actually need to assume the above hypothesis \eqref{asHyp}. It is to be noted that both hypotheses are guaranteed in the presence of a common fixed point of $T$ and $U$, as will be later discussed in more detail (see Lemma \ref{lem:stochastic:bounds}).

In any case, under the assumption \eqref{Hyp}, we can immediately derive a bound on the expectation of the discrete velocity and utilize that to derive our first rate of asymptotic regularity:

\begin{theorem}\label{res:x:ar}
Let $(x_n),(y_n)$ be the sequences generated by \eqref{sHM}. Suppose that $\sum_{n=0}^\infty \alpha_n=\infty$ with rate of divergence $\theta$, that
\[
\sum_{n=0}^\infty \EE[\norm{\xi_{n+1}-\xi_n}], \sum_{n=0}^\infty\EE[\norm{\delta_{n+1}-\delta_n}],\sum_{n=0}^\infty|\alpha_{n+1}-\alpha_n|,\sum_{n=0}^\infty|\beta_{n+1}-\beta_n|<\infty
\]
with rates of convergence $\chi_1$ -- $\chi_4$ and upper bounds $B_1$ -- $B_4$, respectively, and that $\EE[\norm{\xi_n}]\leq E_0$ and $\ \EE[\norm{\delta_n}]\leq D_0$ for all $n\in\NN$. Then $\EE[\norm{x_{n+1}-x_n}]\to 0$ with rate $\varphi_{K,\theta,\chi}$ as well as $\norm{x_{n+1}-x_n}\to 0$ almost surely with rate $\psi_{K,\theta,\chi}$ with $\varphi$, $\psi$ defined as in Lemma \ref{stochasticbabyxu}, i.e.
\[
\varphi_{K,\theta,\chi}(\varepsilon):=\theta\left(\chi\left(\frac{\varepsilon}{2}\right),\ln \left(\frac{2K}{\varepsilon}\right)\right)+1 \text{ and }
\psi_{K,\theta,\chi}(\lambda,\varepsilon):=\varphi_{K,\theta,\chi}\left(\frac{\lambda\varepsilon}{2}\right),
\]
where
\[
\chi(\varepsilon):=\max\{\chi_1(\varepsilon/4),\chi_2(\varepsilon/4),\chi_3(\varepsilon/4(E_0+K_0),\chi_4(\varepsilon/4(D_0+K_0))\}
\]
as well as $K:= 2K_0+E_0+D_0+B$ for $B:=B_1+B_2+B_3(E_0+K_0)+B_4(D_0+K_0)$.
\end{theorem}
\begin{proof}
Using (\ref{eqn:y-to-x}) and (\ref{eqn:x-to-y}) of Lemma \ref{res:recurrence:y-to-x-to-y}, we have that
\[
\norm{x_{n+2}-x_{n+1}}\leq (1-\alpha_{n+1})\norm{x_{n+1}-x_n}+c_n
\]
for all $n\in\mathbb{N}$ everywhere on $\Omega$ where
\begin{align*}
c_n:=&\norm{\xi_{n+1}-\xi_n}+|\alpha_{n+1}-\alpha_n|(\norm{Tx_n-u}+\norm{\xi_n})\\
&+\norm{\delta_{n+1}-\delta_n}+|\beta_{n+1}-\beta_n|(\norm{Uy_n-y_n}+\norm{\delta_n}).
\end{align*}
It is immediate that
\[
\EE[c_n]\leq \EE[\norm{\xi_{n+1}-\xi_n}]+\EE[\norm{\delta_{n+1}-\delta_n}]+|\alpha_{n+1}-\alpha_n|(K_0+E_0)+|\beta_{n+1}-\beta_n|(K_0+D_0)
\]
and so 
\[
\chi(\varepsilon):=\max\{\chi_1(\varepsilon/4),\chi_2(\varepsilon/4),\chi_3(\varepsilon/4(E_0+K_0)),\chi_4(\varepsilon/4(D_0+K_0))\}
\]
is a rate of convergence for $\sum_{n=0}^\infty \EE[c_n]<\infty$, while $B$ as defined above is an upper bound for $\sum_{n=0}^\infty \EE[c_n]$. Naturally, the above yields
\[
\EE[\norm{x_{n+1}-x_{n}}]\leq \EE[\norm{x_{1}-x_{0}}]+\sum_{i=0}^{n-1} \EE[c_i]\leq \EE[\norm{x_{1}-x_{0}}]+B
\]
and we can then show that
\[
\EE[\norm{x_{1}-x_{0}}]\leq \EE[\norm{y_0-u}]+\EE[\norm{Uu-u}]+\EE[\norm{\xi_0}]+\EE[\norm{\delta_0}]\leq 2K_0+E_0+D_0
\]
so that $\EE[\norm{x_{n+1}-x_{n}}]\leq K:= 2K_0+E_0+D_0+B$. Using Lemma \ref{stochasticbabyxu}, we then get the desired rates.
\end{proof}

We can then immediately transfer that rate to the complementary sequence $(y_n)$:

\begin{theorem}\label{res:y:ar}
Under the assumptions of Theorem \ref{res:x:ar}, we have that $\EE[\norm{y_{n+1}-y_n}]\to 0$ with rate
\[
\varphi'(\varepsilon):=\max\{\varphi(\varepsilon/3),\chi_1(\varepsilon/3),\chi_3(\varepsilon/3(K_0+E_0))\}.
\]
as well as $\norm{y_{n+1}-y_n}\to 0$ almost surely with rate
\[
\psi'(\lambda,\varepsilon):=\max\left\{\varphi'(\lambda\varepsilon/2),\chi(\lambda\varepsilon/2)\right\}
\]
where $\varphi$ is a rate for $\EE[\norm{x_{n+1}-x_n}]\to 0$ and $\chi$ is as in Theorem \ref{res:x:ar}.
\end{theorem}
\begin{proof}
Using (\ref{eqn:y-to-x}) of Lemma \ref{res:recurrence:y-to-x-to-y}, we get that 
\[
\norm{y_{n+1}-y_n}\leq\norm{x_{n+1}-x_n}+\norm{\xi_{n+1}-\xi_n}+\vert\alpha_{n+1}-\alpha_n\vert(\norm{Tx_n-u}+\norm{\xi_n})
\]
everywhere on $\Omega$ and for any $n\in\mathbb{N}$. Under expectation, we thus have
\[
\EE[\norm{y_{n+1}-y_n}]\leq\EE[\norm{x_{n+1}-x_n}]+\EE[\norm{\xi_{n+1}-\xi_n}]+\vert\alpha_{n+1}-\alpha_n\vert(K_0+E_0)
\]
and from that the rate for $\EE[\norm{y_{n+1}-y_n}]\to 0$ immediately follows, noting that a rate of convergence $\chi$ for a series $\sum_{n=0}^\infty a_n<\infty$ yields that $\sum_{n=\chi(\varepsilon)}^\infty a_n<\varepsilon$ and so implies that $a_n<\varepsilon$ for any $n\geq\chi(\varepsilon)$. For the rate of $\norm{y_{n+1}-y_n}\to 0$ almost surely, note that using both (\ref{eqn:y-to-x}) and (\ref{eqn:x-to-y}) of Lemma \ref{res:recurrence:y-to-x-to-y}, we get that
\[
\norm{y_{n+2}-y_{n+1}}\leq\norm{y_{n+1}-y_{n}}+d_n
\]
for all $n\in\mathbb{N}$ everywhere on $\Omega$ where
\begin{align*}
d_n:=&\norm{\xi_{n+2}-\xi_{n+1}}+|\alpha_{n+2}-\alpha_{n+1}|(\norm{Tx_{n+1}-u}+\norm{\xi_{n+1}})\\
&+\norm{\delta_{n+1}-\delta_n}+|\beta_{n+1}-\beta_n|(\norm{Uy_n-y_n}+\norm{\delta_n}).
\end{align*}
So, it is immediate that we have 
\[
\EE[d_n]\leq \EE[\norm{\xi_{n+2}-\xi_{n+1}}]+\EE[\norm{\delta_{n+1}-\delta_n}]+|\alpha_{n+2}-\alpha_{n+1}|(K_0+E_0)+|\beta_{n+1}-\beta_n|(K_0+D_0).
\]
and so that $\chi$ from Theorem \ref{res:x:ar} is a rate of convergence for $\sum_{n=0}^\infty \EE[d_n]<\infty$ (noting that if $\chi$ is a rate of convergence for $\sum_{n=0}^\infty a_n<\infty$, then $\sum_{n=\chi(\varepsilon)}a_{n+1}=\sum_{n=\chi(\varepsilon)+1}a_n\leq \sum_{n=\chi(\varepsilon)}a_n<\varepsilon$ so that $\chi$ is also a rate of convergence for $\sum_{n=0}^\infty a_{n+1}<\infty$). Using Lemma \ref{prob:to:as}, we get the desired rate for $\norm{y_{n+1}-y_n}\to 0$ almost surely.
\end{proof}

\subsection{Asymptotic regularity relative to the mappings}

We now move on to establishing rates of asymptotic regularity for the iterations relative to the mappings. For that, we will actually see a crucial dichotomy, where results based on the use of just one of the mappings $U$ or $T$ are comparatively straightforward, whereas for the general case where neither $U$ nor $T$ trivialize, we rely on a geometric argument for establishing a rate of asymptotic regularity relative to $U$ for the sequence $(y_n)$ which requires a uniform convexity assumption on the space $X$, and in the case of asymptotic regularity in mean further relies on an essential use of the uniform integrability of $(\norm{Uy_n-y_n})$ (which for example immediately derives from the assumption \eqref{asHyp}, as also discussed in detail later). As such, before we move on to these results, we first give rates of asymptotic regularity relative to the mappings for the remaining cases, dependent on the relevant rates for $(\norm{Uy_n-y_n})$. As all proofs outside this geometric construction are rather routine, they are deferred to the appendix.

\begin{theorem}\label{res:other:ar}
Assume that $\EE[\norm{Uy_n-y_n}]\to 0$ with rate $\varphi$, and $\EE[\norm{\xi_n}],\EE[\norm{\delta_n}],\alpha_n\to 0$ with rates $\rho_1$ -- $\rho_3$, respectively. Assume further that $\EE[\norm{x_{n+1}-x_n}]\to 0$ with a rate $\varphi_0$. Then
\begin{enumerate}[(a)]
\item $\EE[\norm{x_n-y_n}]\to 0$ with rate
\[
\varphi_1(\varepsilon):=\max\left\{\varphi_0(\varepsilon/3),\varphi(\varepsilon/3),\rho_2(\varepsilon/3)\right\},
\]
\item $\EE[\norm{Ty_n-y_n}]\to 0$ with rate
\[
\varphi_2(\varepsilon):=\max\left\{\varphi_1(\varepsilon/3),\rho_3(\varepsilon/3K_0),\rho_1(\varepsilon/3)\right\},
\]
\item $\EE[\norm{Ux_n-x_n}]\to 0$ with rate
\[
\varphi_3(\varepsilon):=\max\left\{\varphi_1(\varepsilon/3),\varphi(\varepsilon/3)\right\},
\]
\item $\EE[\norm{Tx_n-x_n}]\to 0$ with rate
\[
\varphi_4(\varepsilon):=\max\left\{\varphi_1(\varepsilon/3),\varphi_2(\varepsilon/3)\right\}.
\]        
\end{enumerate}
\end{theorem}

We then can similarly give rates of asymptotic regularity almost surely under a slight extension of the previous conditions on the errors:

\begin{theorem}\label{res:other:ar:as}
Under the assumption \eqref{asHyp}, assume that $\norm{Uy_n-y_n}\to 0$ almost surely with rate $\psi$, and $\alpha_n\to 0$ with rate $\rho$. Further, assume that $\norm{\xi_n},\norm{\delta_n}\to 0$ almost surely with rates $\phi_1,\phi_2$, respectively, and that $\norm{x_{n+1}-x_n}\to 0$ almost surely with a rate $\psi_0$. Then
\begin{enumerate}[(a)]
\item $\norm{x_n-y_n}\to 0$ almost surely with rate
\[
\psi_1(\lambda,\varepsilon):=\max\left\{\psi_0(\lambda/3,\varepsilon/3),\psi(\lambda/3,\varepsilon/3),\phi_2(\lambda/3,\varepsilon/3)\right\},
\]
\item $\norm{Ty_n-y_n}\to 0$ almost surely with rate
\[
\psi_2(\lambda,\varepsilon):=\max\left\{\psi_1(\lambda/3,\varepsilon/3),\rho(\varepsilon\lambda/9K_0),\phi_1(\lambda/3,\varepsilon/3)\right\},
\]
\item $\norm{Ux_n-x_n}\to 0$ almost surely with rate
\[
\psi_3(\lambda,\varepsilon):=\max\left\{\psi_1(\lambda/2,\varepsilon/3),\psi(\lambda/2,\varepsilon/3)\right\},
\]
\item $\norm{Tx_n-x_n}\to 0$ almost surely with rate
\[
\psi_4(\lambda,\varepsilon):=\max\left\{\psi_1(\lambda/2,\varepsilon/3),\psi_2(\lambda/2,\varepsilon/3)\right\}.
\]        
\end{enumerate}
\end{theorem}

\begin{remark}\label{rem:sumExpAS}
If we assume that $\sum_{n=0}^\infty \EE[\norm{\xi_n}]$, $\sum_{n=0}^\infty \EE[\norm{\delta_n}]<\infty$ with rates of convergence $\chi_1$, $\chi_2$, respectively, then we can immediately derive rates $\phi_1,\phi_2$ for $\norm{\xi_n},\norm{\delta_n}\to 0$ almost surely: Since $\sum_{n=0}^\infty \EE[\norm{\xi_n}]<\infty$ with rate $\chi_1$, we get
\[
\PP(\exists n\geq \chi_1(\lambda\varepsilon)(\norm{\xi_n}\geq\varepsilon))\leq\sum_{n=\chi_1(\lambda\varepsilon)}^\infty\PP(\norm{\xi_n}\geq\varepsilon)\leq\sum_{n=\chi_1(\lambda\varepsilon)}^\infty\frac{\EE[\norm{\xi_n}]}{\varepsilon}<\lambda
\]
using Markov's inequality so that $\phi_1(\lambda,\varepsilon):=\chi_1(\lambda\varepsilon)$ is a rate for $\norm{\xi_n}\to 0$ almost surely. Similarly for $\sum_{n=0}^\infty \EE[\norm{\delta_n}]<\infty$ and $\chi_2,\phi_2$.
\end{remark}

\subsubsection{Special cases of the Halpern iteration and the Krasnoselskii-Mann iteration with Tikhonov regularization terms}

In the special case of $U:=\mathrm{Id}$ and $\delta_n:=0$, the iteration \eqref{sHM} collapses to the stochastic Halpern iteration \eqref{sH}. We then have trivial rates for $\EE[\norm{Uy_n-y_n}]\to 0$ and $\norm{Uy_n-y_n}\to 0$ almost surely and so, in that case, we get under the assumptions of Theorems \ref{res:x:ar} and \ref{res:other:ar} (and also under suitable monotonicity assumptions of the rates involved) that
\[
\EE[\norm{Tx_n-x_n}]\to 0
\]
with a rate
\[
\varphi(\varepsilon):=\max\{\varphi_{K,\theta,\chi}(\varepsilon/27),\rho_3(\varepsilon/9K_0),\rho_1(\varepsilon/9)\}
\]
with $\varphi_{K,\theta,\chi}$ defined as in Theorem \ref{res:x:ar} and $\rho_1,\rho_3$ as in Theorem \ref{res:other:ar}. Note that this generalises known rates in this case \cite[Theorem 3.3]{BC2024}, which apply only to specific choices of the parameters. In a similar way, we get a new rate for $\norm{Tx_n-x_n}\to 0$ almost surely, though we do not spell it out here.

In the special case of $T:=\mathrm{Id}$ and $\xi_n:=0$, the iteration \eqref{sHM} collapses to a stochastic variant of the Krasnoselskii-Mann iteration with Tikhonov regularization terms \eqref{sKMT} (and even a slight extension by allowing general anchors $u$). In that case, we do not need to rely on the geometric arguments discussed in the next part of this section and can instead directly derive rates of convergence for $\EE[\norm{Uy_n-y_n}]\to 0$ and $\norm{Uy_n-y_n}\to 0$ almost surely, essentially following the approach of \cite{SK2012} (see also \cite{CKL2023}).

\begin{lemma}[essentially \cite{SK2012}]\label{lem:TMineq}
Let $(x_n),(y_n)$ be the sequences generated by \eqref{sHM} where $T=\mathrm{Id}$ and $\xi_n=0$ for all $n\in\mathbb{N}$. Then the following recurrence relation holds everywhere on $\Omega$ for all $n\in\mathbb{N}$:
\[
\norm{Uy_{n+1}-y_{n+1}}\leq 2\norm{y_n-y_{n+1}}+\alpha_{n+1}\norm{Uy_{n+1}-u}+\norm{\delta_n}+\beta_n\norm{Uy_{n+1}-y_{n+1}}.
\]
\end{lemma}

From that inequality, the following rates follow in a straightforward way:

\begin{theorem}\label{thm:UAsRegKMT}
Let $(x_n),(y_n)$ be the sequences generated by \eqref{sHM} where $T=\mathrm{Id}$ and $\xi_n=0$ for all $n\in\mathbb{N}$. Also, let $\Lambda>0$ be such that $\Lambda\leq\beta_n\leq 1-\Lambda$ for all $n\in\NN$. If $\EE[\norm{y_{n}-y_{n+1}}]\to 0$ with rate $\varphi$, $\alpha_n\to 0$ with rate $\rho$ and $\EE[\norm{\delta_n}]\to 0$ with rate $\chi$, then $\EE[\norm{Uy_n-y_n}]\to 0$ with rate
\[
\kappa(\varepsilon):=\max\{\varphi(\Lambda\varepsilon/4),\rho(\Lambda\varepsilon/4K_0),\chi(\Lambda\varepsilon/4)\}+1.
\]
Under the alternative hypothesis \eqref{asHyp} and assuming $\norm{y_{n}-y_{n+1}}\to 0$ almost surely with rate $\psi$, $\alpha_n\to 0$ with rate $\rho$ and $\norm{\delta_n}\to 0$ almost surely with rate $\phi$, then $\norm{Uy_n-y_n}\to 0$ almost surely with rate
\[
\zeta(\lambda,\varepsilon):=\max\{\psi(\lambda/3,\Lambda\varepsilon/4),\rho(\Lambda\lambda\varepsilon/4K_0),\phi(\lambda/3,\Lambda\varepsilon/4)\}+1.
\]
\end{theorem}

\subsubsection{The general case}

We now discuss an alternative scenario where, in particular, a random variable $Y$ satisfying $\eqref{asHyp}$ can be explicitly constructed if our mappings possess a common fixed point. To be more precise, let us assume that $\mathrm{Fix}T\cap\mathrm{Fix}U\neq\emptyset$ and that $p$ is a common fixed point of $T$ and $U$. Further, instead of making the assumptions \eqref{Hyp} or $\eqref{asHyp}$, for the rest of this section we now fix
\[
D\geq \sum_{n=0}^\infty\EE[\norm{\delta_n}]\text{ and }E\geq \sum_{n=0}^\infty\EE[\norm{\xi_n}]
\]
as well as a $K_0$ such that $K_0\geq \EE[\norm{x_0-p}],\EE[\norm{u-p}]$. Using these data, we immediately get the following extended result on bounds:

\begin{lemma}\label{lem:stochastic:bounds}
For all $n\in\NN$, $\norm{x_n-p}\leq Y'\leq 2Y'=:Y$ pointwise everywhere, where
\[
Y':=\norm{x_0-p}+\norm{u-p}+\sum_{i=0}^\infty \left(\norm{\xi_n}+\norm{\delta_n}\right)
\]
and furthermore $\EE[Y]\leq K:=4K_0+2D+2E$. The sequences
\[
\norm{y_n-p},\norm{x_{n+1}-x_n},\norm{y_{n+1}-y_n},\norm{Tx_n-u},\norm{Uy_n-y_n},\norm{Uy_n-u}
\]
are ``$L_1$-dominated" by $Y$ in a similar way.
\end{lemma}
\begin{proof}
Pointwise everywhere it holds that
\begin{align*}
\norm{x_{n+1}-p}&=\norm{(1-\beta_n)(Uy_n+\delta_n)+\beta_ny_n-p}\\
&\leq (1-\beta_n)\norm{Uy_n-p}+\beta_n\norm{y_n-p}+\norm{\delta_n}\\
&\leq \norm{y_n-p}+\norm{\delta_n}\\
&=\norm{(1-\alpha_n)(Tx_n+\xi_n)+\alpha_n u-p}+\norm{\delta_n}\\
&\leq (1-\alpha_n)\norm{Tx_n-p}+\alpha_n\norm{u-p}+\norm{\xi_n}+\norm{\delta_n}\\
&\leq (1-\alpha_n)\norm{x_n-p}+\alpha_n\norm{u-p}+\norm{\xi_n}+\norm{\delta_n}.
\end{align*}
It follows immediately by induction that
\[
\norm{x_{n+1}-p}\leq Y'_n:=\norm{x_0-p}+\norm{u-p}+\sum_{i=0}^n \left(\norm{\xi_n}+\norm{\delta_n}\right)
\]
holds pointwise everywhere. Since the $Y'_n$ are pointwise monotone, defining $Y':=\sup_{n\in\NN}Y'_n$ yields that $\norm{x_n-p}\leq Y'_n\leq Y'$ pointwise everywhere for all $n\in\NN$, and by the monotone convergence theorem we have
\[
\EE[Y']=\EE[\norm{x_0-p}]+\EE[\norm{u-p}]+\sum_{i=0}^\infty \left(\EE[\norm{\xi_n}]+\EE[\norm{\delta_n}]\right)\leq 2K_0+D+E.
\]
Therefore immediately $\EE[Y]\leq K$. By the above inequalities, one also has $\norm{y_n-p}\leq Y'_n\leq Y'$, and the rest of the bounds follow by the triangle inequality.
\end{proof}

For the rest of this section, we will always assume the existence of a fixed point as above and use $Y$ and $K$ to refer to the quantities in Lemma \ref{lem:stochastic:bounds}. Note that these in particular validate the assumptions \eqref{Hyp} and \eqref{asHyp}.

We now move on to the asymptotic regularity relative to $U$ of the sequence $(y_n)$, which then in conjunction with Theorems \ref{res:other:ar} and \ref{res:other:ar:as} will allow us to derive all of the other regularity properties relative to mappings. For that, we initially establish $\norm{Uy_n-y_n}\to 0$ almost surely using geometric properties of the underlying space. Using that $(\norm{Uy_n-y_n})$ is uniformly integrable, as follows from \eqref{asHyp} and hence from Lemma \ref{lem:stochastic:bounds} in the context of the present assumptions, we can then establish the asymptotic regularity relative to $U$ of $(y_n)$ in expectation. The quantitative result will then in particular depend on quantitative renderings of both the uniform convexity of the space and the uniform integrability of the sequence. We begin with the geometric assumption on the underlying normed space $(X,\norm{\cdot})$:

\begin{definition}
We say that $(X,\norm{\cdot})$ is uniformly convex (\cite{Cla1936}) if for any $\varepsilon\in (0,2]$, there exists a $\delta\in (0,1]$ such that for all $x,y\in\overline{B}_1(0)$:
\[
\norm{x-y}\geq\varepsilon\text{ implies }\norm{\frac{x+y}{2}}\leq 1-\delta.
\]
We call a modulus $\eta:(0,2]\to (0,1]$ witnessing such a $\delta$ in terms of $\varepsilon$ a modulus of uniform convexity for $X$.
\end{definition}

We will later discuss examples of such moduli of uniform convexity for particular (classes of) spaces (see in particular Lemma \ref{lem:ippPowerType}). The above modulus also applies to closed balls of any radius centered at any point in the space and for arbitrary convex combinations:

\begin{lemma}\label{lem:modUniConvFull}
Let $\eta$ be a modulus of uniform convexity. For any $r>0$ and $\varepsilon\in (0,2]$, if $x,y\in \overline{B}_r(a)$ for $a\in X$ with $\norm{x-y}\geq\varepsilon\cdot r$, then for all $\lambda\in [0,1]$:
\[
\norm{(1-\lambda)x+\lambda y-a}\leq (1-2\lambda(1-\lambda)\eta(\varepsilon))r.
\]
\end{lemma}

The proof is straightforward and we hence omit it (but refer e.g.\ \cite{Leu2010} for a proof of such a property even in the context of nonlinear uniformly convex hyperbolic spaces).

The proof of the following theorem now follows the outline of the proof of an analogous result for the Halpern-Mann iteration in uniformly convex hyperbolic spaces as given in \cite{LP2024} (though without errors, even nonstochastic ones):

\begin{theorem}\label{geometry:as}
Let $(X,\norm{\cdot})$ be uniformly convex with modulus $\eta$. Let $\norm{x_{n+1}-x_n}\to 0$ almost surely with rate $\Delta$. Also, let $\rho$ be a rate for $\alpha_n\to 0$ and assume that $\sum_{n=0}^\infty \EE[\norm{\xi_n}]$, $\sum_{n=0}^\infty \EE[\norm{\delta_n}]<\infty$ with rates of convergence $\chi_1$, $\chi_2$, respectively. Lastly, let $\Lambda>0$ be such that $\Lambda\leq\beta_n\leq 1-\Lambda$. Then $\norm{Uy_n-y_n}\to 0$ almost surely with rate
\[
\Gamma(\lambda,\varepsilon):=\max\{\Delta(\lambda/9,\widehat{\varepsilon}/4),\rho(\widehat{\varepsilon}/4K'),\chi_1(\lambda\widehat{\varepsilon}/36),\chi_2(\lambda\widehat{\varepsilon}/36)\}
\]
for $\widehat{\varepsilon}:=\varepsilon\cdot\Lambda^2\cdot \eta(\varepsilon/K')$ and $K':=3K/\lambda$.
\end{theorem}
\begin{proof}
Suppose for contradiction that
\[
\PP\left(\exists n\geq \Gamma(\lambda,\varepsilon)\, (\norm{Uy_n-y_n}\geq\varepsilon)\right)\geq \lambda
\]
and call the set inside the probability $B_{\lambda,\varepsilon}$. By Lemma \ref{lem:stochastic:bounds} and Markov's inequality, we have
\[
\PP\left(\exists n\left(\norm{y_n-p}> \frac{K}{\lambda}\right)\right)\leq \PP\left(Y\geq \frac{K}{\lambda}\right)\leq \frac{\EE[Y]}{K/\lambda}\leq \lambda
\]
for any $\lambda>0$ and so $\PP\left(\exists n \left(\norm{y_n-p}> K'\right)\right)\leq \lambda/3$ for $K':=3K/\lambda$. Similarly for $\norm{u-p}$. Thus, using the Fr\'echet inequalities, we have
\begin{align*}
&\PP(\exists n\geq \Gamma(\lambda,\varepsilon)(\norm{Uy_n-y_n}\geq \varepsilon)\land \forall n\left(\norm{y_n-p},\norm{u-p}\leq K'\right))\\
&\qquad\geq \PP(\exists n\geq \Gamma(\lambda,\varepsilon)(\norm{Uy_n-y_n}>\varepsilon)) + \PP(\forall n(\norm{y_n-p}\leq K')) + \PP(\forall n(\norm{u-p}\leq K')) - 2\\
&\qquad \geq \lambda+(1-\lambda/3)+ (1-\lambda/3)-2=\lambda/3.
\end{align*}
We denote that set measured in the above by $A_{\lambda,\varepsilon}$, and let $\omega\in A_{\lambda,\varepsilon}$, i.e.\ there exists some $n(\omega)\geq\Gamma(\lambda,\varepsilon)$ such that
\[
\norm{Uy_{n(\omega)}(\omega)-y_{n(\omega)}(\omega)}\geq\varepsilon\text{ and } \norm{y_{n(\omega)}(\omega)-p},\norm{u(\omega)-p}\leq K'.
\]
Writing $n_0$ for $n(\omega)$, we then have $\norm{Uy_{n_0}(\omega)-y_{n_0}(\omega)}\leq 2\norm{y_{n_0}(\omega)-p}\leq 2K'$ so that $\varepsilon/2\leq \norm{y_{n_0}(\omega)-p}\leq K'$. Also, we have $\norm{Uy_{n_0}(\omega)-p}\leq\norm{y_{n_0}(\omega)-p}\leq K'$ as well as
\[
\norm{Uy_{n_0}(\omega)-y_{n_0}(\omega)}\geq \varepsilon=\varepsilon/K'\cdot K'\geq \varepsilon/K'\norm{y_{n_0}(\omega)-p}
\]
and $\varepsilon/K'\leq 2$. So, we can apply Lemma \ref{lem:modUniConvFull} to derive
\begin{align*}
\norm{x_{{n_0}+1}(\omega)-p}&=\norm{(1-\beta_{n_0})(Uy_{n_0}(\omega)+\delta_{n_0}(\omega))+\beta_{n_0}y_{n_0}(\omega)-p}\\
&\leq \norm{(1-\beta_{n_0})Uy_{n_0}(\omega)+\beta_{n_0}y_{n_0}(\omega)-p}+\norm{\delta_{n_0}(\omega)}\\
&\leq (1-2\beta_{n_0}(1-\beta_{n_0})\eta(\varepsilon/K'))\norm{y_{n_0}(\omega)-p}+\norm{\delta_{n_0}(\omega)}\\
&\leq \norm{y_{n_0}(\omega)-p}-2\norm{y_{n_0}(\omega)-p}\Lambda^2\eta(\varepsilon/K')+\norm{\delta_{n_0}(\omega)}\\
&\leq\norm{y_{n_0}(\omega)-p}-\varepsilon\cdot\Lambda^2\cdot\eta(\varepsilon/K')+\norm{\delta_{n_0}(\omega)}.
\end{align*}
Now, we further have
\begin{align*}
\norm{y_{n_0}(\omega)-p}&\leq (1-\alpha_{n_0})\norm{Tx_{n_0}(\omega)-p}+\alpha_{n_0}\norm{u(\omega)-p}+\norm{\xi_{n_0}(\omega)}\\
&\leq \norm{x_{n_0}(\omega)-p}+\alpha_{n_0}\norm{u(\omega)-p}+\norm{\xi_{n_0}(\omega)}
\end{align*}
so that we can in particular derive
\[
\norm{x_{{n_0}+1}(\omega)-p}\leq\norm{x_{n_0}(\omega)-p}+\alpha_{n_0} K'+\norm{\xi_{n_0}(\omega)}+\norm{\delta_{n_0}(\omega)}-\varepsilon\cdot\Lambda^2\cdot\eta(\varepsilon/K').
\]
So, in the end we have
\begin{align*}
\widehat{\varepsilon}&=\varepsilon\cdot \Lambda^2 \cdot \eta(\varepsilon/K')\\
&\leq\norm{x_{n_0}(\omega)-p}-\norm{x_{{n_0}+1}(\omega)-p}+\alpha_{n_0}K'+\norm{\xi_{n_0}(\omega)}+\norm{\delta_{n_0}(\omega)}\\
&\leq \norm{x_{{n_0}+1}(\omega)-x_{n_0}(\omega)}+\alpha_{n_0}K'+\norm{\xi_{n_0}(\omega)}+\norm{\delta_{n_0}(\omega)}.
\end{align*}
Letting $V_{n}:=\norm{x_{n+1}-x_n}+\alpha_nK'+\norm{\xi_n}+\norm{\delta_n}$, we have shown that
\[
A_{\lambda,\varepsilon}\subseteq \{\exists n\geq \Gamma(\lambda,\varepsilon)\, (V_n\geq\widehat{\varepsilon})\}.
\]
Similarly to in the proof of Theorem \ref{res:other:ar:as}, we now have that $\chi_1(\lambda\varepsilon),\chi_2(\lambda\varepsilon)$ are rates for $\norm{\xi_n},\norm{\delta_n}\to 0$, respectively. So we have
\begin{align*}
\lambda/3&\leq\PP(A_{\lambda,\varepsilon})\\
&\leq\PP(\exists n\geq \Gamma(\lambda,\varepsilon)\, (V_{n}\geq \widehat{\varepsilon}))\\
&\leq \PP\left(\exists n\geq \Gamma(\lambda,\varepsilon)\, \left((\norm{x_{n+1}-x_n}\geq \widehat{\varepsilon}/4)\cup (\alpha_nK'\geq \widehat{\varepsilon}/4)\cup (\norm{\xi_n}\geq \widehat{\varepsilon}/4)\cup (\norm{\delta_n}\geq \widehat{\varepsilon}/4) \right)\right)\\
&\leq \PP\left(\exists n\geq \Delta(\lambda/9,\widehat{\varepsilon}/4)(\norm{x_{n+1}-x_n}\geq \widehat{\varepsilon}/4)\right)+\PP\left(\exists n\geq \rho(\widehat{\varepsilon}/4K')(\alpha_nK'\geq \widehat{\varepsilon}/4)\right)\\
&\quad+\PP\left(\exists n\geq \chi_1(\lambda\widehat{\varepsilon}/36)(\norm{\xi_n}\geq \widehat{\varepsilon}/4)\right)+\PP\left(\exists n\geq \chi_2(\lambda\widehat{\varepsilon}/36)(\norm{\delta_n}\geq \widehat{\varepsilon}/4)\right)\\
&<\frac{\lambda}{9}+0+\frac{\lambda}{9}+\frac{\lambda}{9}=\frac{\lambda}{3},
\end{align*}
a contradiction.
\end{proof}

\begin{remark}\label{rem:arUYOpt}
Using a slightly different argument first devised in \cite[Theorem 3.4]{Koh2003} (see also \cite[Remark 15]{Leu2007} or \cite[Remark 3.7]{LP2024} for similar remarks in the context of nonlinear spaces), we can slightly optimize the above rate in the context of moduli of uniform convexity of a special form: Let $\eta(\varepsilon)=\varepsilon\cdot\tilde\eta(\varepsilon)$ where $\tilde\eta$ is increasing. Then the above rate $\Gamma$ holds even with $\widehat{\varepsilon}:=\varepsilon\cdot\Lambda^2\cdot \tilde\eta(\varepsilon/K')$.

To see this, follow the proof of Theorem \ref{geometry:as} but replace $\varepsilon/K'$ with $\varepsilon/\norm{y_{n_0}(\omega)-p}$. Then also $\varepsilon/\norm{y_{n_0}(\omega)-p}\leq 2$ as well as 
\[
\norm{Uy_{n_0}(\omega)-y_{n_0}(\omega)}\geq \varepsilon/\norm{y_{n_0}(\omega)-p}\cdot \norm{y_{n_0}(\omega)-p}
\]
and this leads to
\begin{align*}
\norm{x_{{n_0}+1}(\omega)-p}&\leq \norm{y_{n_0}(\omega)-p}-2\cdot\varepsilon\cdot\Lambda^2\cdot\tilde\eta(\varepsilon/\norm{y_{n_0}(\omega)-p})+\norm{\delta_{n_0}(\omega)}\\
&\leq\norm{y_{n_0}(\omega)-p}-\varepsilon\cdot\Lambda^2\cdot\tilde\eta(\varepsilon/\norm{y_{n_0}(\omega)-p})+\norm{\delta_{n_0}(\omega)}\\
&\leq\norm{y_{n_0}(\omega)-p}-\varepsilon\cdot\Lambda^2\cdot\tilde\eta(\varepsilon/K')+\norm{\delta_{n_0}(\omega)}.
\end{align*}
using that $\tilde\eta(\varepsilon/K')\leq\tilde\eta(\varepsilon/\norm{y_{n_0}(\omega)-p})$ as $\tilde\eta(\varepsilon/\norm{y_{n_0}(\omega)-p})\leq K'$ and since $\tilde\eta$ is increasing. Then the proof continuous as before.
\end{remark}

We now discuss the assumptions from the quantitative theory of expected values that we require to establish an analogous result on the asymptotic regularity relative to $U$ of $(y_n)$ in mean.

\begin{definition}
\label{def:modulus:cont}
Let $X$ be an integrable random variable. We call a function $\mu:(0,\infty)\to (0,\infty)$ such that
\[
\forall \varepsilon>0\,\forall A\in \mathcal{F}\left( \PP(A)\leq\mu(\varepsilon)\to \EE[\vert X\vert 1_A]\leq\varepsilon\right)
\]
a modulus of absolute continuity for $X$.
\end{definition}

By e.g.\ Lemma 13.1 in \cite{Wil1991}, such a modulus always exists for integrable $X$. However, it provides a quantitatively different rendering of that property than a simple upper bound on the mean and can in general not be derived from such a bound. We begin with a simple property of such moduli.

\begin{lemma}\label{lem:absContModBoundLem}
Let $X$ be an integrable random variable and $\mu$ a modulus of absolute continuity for $X$. For any $a,\varepsilon\in (0,\infty)$, we have that
\[
\EE[\vert X\vert]\geq a+\varepsilon\text{ implies }\PP(\vert X\vert >a)>\mu(\varepsilon/2).
\]
In particular, we have that $\EE[\vert X\vert]\geq \varepsilon$ implies $\PP(\vert X\vert>\varepsilon/2)>\mu(\varepsilon/4)$.
\end{lemma}
\begin{proof}
Suppose $\PP(\vert X\vert>a)\leq\mu(\varepsilon/2)$. Then we have
\[
\EE[\vert X\vert]=\EE[\vert X\vert 1_{\vert X\vert\leq a}]+\EE[\vert X\vert 1_{\vert X\vert> a}]\leq a+\varepsilon/2 < a+\varepsilon
\]
which is the claim.
\end{proof}

A modulus of absolute continuity now motivates the following new notion of a modulus of uniform integrability:

\begin{definition}
\label{uniformint:def}
A sequence of random variables $(X_n)$ is called uniformly integrable if both $\sup_{n\in\mathbb{N}}\EE[\vert X_n\vert]<\infty$ and for any $\varepsilon>0$, there exists a $\delta>0$ such that 
\[
\forall n\in\mathbb{N}\,\forall A\in\mathcal{F}\left( \PP(A)\leq\delta\to \EE[\vert X_n\vert 1_A]\leq\varepsilon\right).
\]
We call a function $\mu$ that witnesses such a $\delta$ in terms of $\varepsilon$ a modulus of uniform integrability for $(X_n)$.
\end{definition}

Note $\mu$ is a modulus of uniform integrability for $(X_n)$ exactly when $\mu$ is a modulus of absolute continuity for any $X_n$. The main use that a modulus of uniform integrability has for a stochastic process is that with it, we can transfer a rate of almost-sure convergence to a rate of convergence in mean:

\begin{lemma}\label{as:ui:mean}
Let $(X_n)$ be a sequence of nonnegative random variables such that $X_n\to 0$ almost surely with rate $\varphi$ and such that $\mu$ is modulus of uniform integrability for $(X_n)$. Then $\EE[X_n]\to 0$ with rate
\[
\Gamma(\varepsilon):=\varphi\left(\mu\left(\frac{\varepsilon}{4}\right),\frac{\varepsilon}{2}\right).
\]
\end{lemma}
\begin{proof}
Suppose for contradiction that there exists some $n_0\geq \Gamma(\varepsilon)$ with $\EE[X_{n_0}]\geq \varepsilon$. Then by Lemma \ref{lem:absContModBoundLem} we have $\PP(X_{n_0}>\varepsilon/2)>\mu(\varepsilon/4)$ and hence $\PP\left(\exists n\geq \Gamma(\varepsilon)(X_n\geq \varepsilon/2)\right)\geq \PP(X_{n_0}>\varepsilon/2)>\mu(\varepsilon/4)$, a contradiction.
\end{proof}

We now obtain the following result on rates of asymptotic regularity in mean, derived from the previous Theorem \ref{geometry:as} using a modulus of uniform integrability for $(\norm{Uy_n-y_n})$.

\begin{theorem}\label{geometry:alt}
Let $(X,\norm{\cdot})$ be uniformly convex with modulus $\eta$. Under the assumptions of Theorem \ref{res:x:ar}, let $\EE[\norm{x_{n+1}-x_n}]\to 0$ with rate $\Delta$ from Theorem \ref{res:x:ar}. Also, let $\rho$ be a rate for $\alpha_n\to 0$ and assume that $\sum_{n=0}^\infty \EE[\norm{\xi_n}]$, $\sum_{n=0}^\infty \EE[\norm{\delta_n}]<\infty$
with rates of convergence $\chi_1$, $\chi_2$, respectively. Also, let $\Lambda>0$ be such that $\Lambda\leq\beta_n\leq 1-\Lambda$. Lastly, let $\mu$ be a modulus of uniform integrability for $(\norm{Uy_n-y_n})$. Then $\EE[\norm{Uy_n-y_n}]\to 0$ with rate
\[
\Gamma(\varepsilon):=\max\{\Delta(\overline{\varepsilon}),\rho(\widehat{\varepsilon}/4K'),\chi_1(\overline\varepsilon),\chi_2(\overline\varepsilon)\}
\]
where $\overline\varepsilon:=\widehat{\varepsilon}\mu(\varepsilon/4)/36$ for $\widehat{\varepsilon}:=\varepsilon/2\cdot\Lambda^2\cdot \eta(\varepsilon/2K')$ and $K':=3K/\mu(\varepsilon/4)$.
\end{theorem}
\begin{proof}
By Theorem \ref{res:x:ar}, $\Delta(\lambda\varepsilon/2)$ is a rate of almost sure convergence for $\norm{x_{n+1}-x_n}\to 0$. Therefore by Theorem \ref{geometry:as}, a rate of almost sure convergence for $\norm{Uy_n-y_n}\to 0$ is given by
\begin{align*}
\tilde\Gamma(\lambda,\varepsilon)&:=\max\{\Delta(\lambda\tilde{\varepsilon}/36),\rho(\tilde{\varepsilon}/4K'),\chi_1(\lambda\tilde{\varepsilon}/36),\chi_2(\lambda\tilde{\varepsilon}/36)\}
\end{align*}
for $\tilde{\varepsilon}:=\varepsilon\cdot\Lambda^2\cdot \eta(\varepsilon/K')$ and $K':=3K/\lambda$. By Lemma \ref{as:ui:mean} we have $\EE[\norm{Uy_n-y_n}]\to 0$ with rate
\[
\Gamma(\varepsilon):=\tilde\Gamma\left(\mu\left(\frac{\varepsilon}{4}\right),\frac{\varepsilon}{2}\right)=\max\{\Delta(\overline{\varepsilon}),\rho(\widehat{\varepsilon}/4K'),\chi_1(\overline\varepsilon),\chi_2(\overline\varepsilon)\}
\]
where $\overline\varepsilon:=\widehat{\varepsilon}\mu(\varepsilon/4)/36$ now for $\widehat{\varepsilon}:=\varepsilon/2\cdot\Lambda^2\cdot \eta(\varepsilon/2K')$ and $K':=3K/\mu(\varepsilon/4)$.
\end{proof}

\begin{remark}\label{rem:arUYOptExp}
Using Remark \ref{rem:arUYOpt}, it follows that also here, if $\eta(\varepsilon)=\varepsilon\cdot\tilde\eta(\varepsilon)$ where $\tilde\eta$ is increasing, then above rate $\Gamma$ holds even with $\widehat{\varepsilon}$ defined as $\widehat{\varepsilon}:=\varepsilon/2\cdot\Lambda^2\cdot \tilde\eta(\varepsilon/2K')$.
\end{remark}

The following result collects the qualitative core of the above asymptotic regularity results for \eqref{sHM}, void of any quantitative considerations, to illustrate the main assumptions more clearly:

\begin{corollary}
\label{cor:qualitative}
Let $(X,\norm{\cdot})$ be uniformly convex and let $(x_n),(y_n)$ be the sequences generated by \eqref{sHM} where $\mathrm{Fix}T\cap\mathrm{Fix}U\neq\emptyset$. Assume $\sum_{n=0}^\infty \EE[\norm{\xi_n}]$, $\sum_{n=0}^\infty \EE[\norm{\delta_n}]<\infty$ and $\EE[\norm{x_0-p}],\EE[\norm{u-p}]<\infty$ together with $\sum_{n=0}^\infty \alpha_n=\infty$ and $\alpha_n\to 0$ as well as
\[
\sum_{n=0}^\infty \EE[\norm{\xi_{n+1}-\xi_n}], \sum_{n=0}^\infty\EE[\norm{\delta_{n+1}-\delta_n}],\sum_{n=0}^\infty|\alpha_{n+1}-\alpha_n|,\sum_{n=0}^\infty|\beta_{n+1}-\beta_n|<\infty
\]
and $\Lambda\leq \beta_n\leq 1-\Lambda$ for some $\Lambda>0$. Then $\norm{Uy_n-y_n}\to 0$ almost surely and in mean.
\end{corollary}

To be more precise, Corollary \ref{cor:qualitative} follows from Theorems \ref{geometry:as} and \ref{geometry:alt} in conjunction with Theorem \ref{res:x:ar}. As emphasised already, the requirement in Theorem \ref{geometry:alt} that $(\norm{Uy_n-y_n})$ comes equipped with a modulus of uniform integrability disappears in the qualitative result: Under the assumptions of this section, \eqref{asHyp} is automatically satisfied, and in particular $(\norm{Uy_n-y_n})$ is dominated by some integrable random variable $Y$ (explicitly definable by Lemma \ref{lem:stochastic:bounds}) and is thus automatically uniformly integrable. The modulus is only required in order to construct our rate of convergence.

Indeed, both Theorem \ref{geometry:as} and \ref{geometry:alt} provide rather complex constructions for the corresponding rates of asymptotic regularity, featuring an interplay of many different moduli, in particular of the moduli of uniform convexity $\eta$ and uniform integrability $\mu$. While concrete instantiations for the modulus $\eta$ for special (classes of) spaces will be discussed later on (see Lemma \ref{lem:ippPowerType}) in the context of standard assumptions on the errors and the parameters which result in particularly fast rates, here we illustrate how a corresponding modulus $\mu$ can be derived in more concrete situations.

We note that under our general assumption \eqref{asHyp} whereby $(\norm{Uy_n-y_n})$ is dominated by some $Y$ with finite mean, a modulus of absolute continuity for $Y$ (in the sense of Definition \ref{def:modulus:cont}) is clearly also a modulus of uniform integrability for the sequence $(\norm{Uy_n-y_n})$, and our quantitative uniform integrability requirement is thus reduced to a quantitative continuity property of the bound $Y$. Here the bounds $D,E$ for the summability assumption on the error terms do not suffice to derive such a modulus explicitly, as they do not allow for a construction of corresponding moduli of absolute continuity for the series, but provided these assumptions are extended to include such moduli of uniform integrability for $\sum_{n=0}^\infty\EE[\norm{\delta_n}]$ and $\sum_{n=0}^\infty\EE[\norm{\xi_n}]$, then we can effectively construct our modulus $\mu$ as follows:

\begin{lemma}
\label{lem:ui:to:errors}
Suppose that $u$ and $x_0$ are chosen to be constant, that $K>0$ is such that $\norm{x_0-p},\norm{u-p}<K$, and in addition both $\sum_{i=0}^\infty \norm{\xi_i}$ and $\sum_{i=0}^\infty \norm{\delta_i}$ are integrable with moduli of absolute continuity $\mu_1$ and $\mu_2$ respectively. Then a modulus of uniform integrability for $(\norm{Uy_n-y_n})$ is given by
\[
\mu(\varepsilon):=\min\left\{\frac{\varepsilon}{8K},\mu_1\left(\frac{\varepsilon}{8}\right),\mu_2\left(\frac{\varepsilon}{8}\right)\right\}.
\] 
\end{lemma}

\begin{proof}
By Lemma \ref{lem:stochastic:bounds} we have $\norm{Uy_n-y_n}\leq Y$ pointwise everywhere for all $n\in\NN$, where $Y$ is defined as in Lemma \ref{lem:stochastic:bounds}, and so in particular we then have
\[
\norm{Uy_n-y_n}\leq 2\left(2K+\sum_{i=0}^\infty \left(\norm{\xi_n}+\norm{\delta_n}\right)\right).
\]
for any $n\in\NN$. Thus for any $A\in\mathcal{F}$ and $n\in\NN$ it follow that
\[
\EE[\norm{Uy_n-y_n}1_A]\leq 4K\PP(A)+2\left(\sum_{i=0}^\infty \left(\EE[\norm{\xi_i}1_A]+\EE[\norm{\delta_i}1_A]\right)\right)
\]
and so the result follows by definition of $\mu_1$ and $\mu_2$.
\end{proof}

While such moduli $\mu_1,\mu_2$ always exist, the question of how these and related moduli of uniform integrability can actually be constructed in concrete situations still remains. First, we observe that in scenarios where $(\norm{Uy_n-y_n})$, or its $L_1$-bound $Y$ or the series over the errors enjoy stronger properties, such as higher moment conditions, corresponding moduli can be given in a straightforward manner and moreover assume a very simple form.

For that, we first consider the following general result:

\begin{lemma}\label{uniformint:lem}
Let $(X_n)$ be a sequence of random variables such that $\sup_{n\in\NN}\EE[g(|X_n|)]<\infty$ for some measurable supercoercive $g:[0,\infty)\to [0,\infty)$, i.e.\ $g(x)/x\to \infty$ as $x\to\infty$. Then $(X_n)$ is uniformly integrable and $\mu$ defined by
\[
\mu(\varepsilon):=\frac{\varepsilon}{2} \left(\kappa\left(\frac{2K}{\varepsilon}\right)\right)^{-1} 
\]
is a corresponding modulus in the sense of Definition \ref{uniformint:def}, where $\sup_{n\in\NN}\EE[g(|X_n|)]<K$ and $\kappa:(0,\infty)\to (0,\infty)$ is a rate of divergence for $g(x)/x\to \infty$, that is
\[
\forall a>0\,\forall x\geq \kappa(a)\left( \frac{g(x)}{x}\geq a\right).
\]
\end{lemma}
\begin{proof}
Uniform integrability of $(X_n)$ is a standard fact \cite{Pou1915} (see also \cite[Theorem 6.19]{Kle2020}), and we simply need to verify the moduli. Fix $\varepsilon>0$, $n\in\NN$ and $A\in \mathcal{F}$. Note that for any $a>0$, we have
\begin{align*}
\EE[\vert X_n \vert 1_A]&\leq \EE[\vert X_n \vert 1_{A\cap (\vert X_n \vert\leq \kappa(a))}]+\EE[|X_n|1_{A\cap (\vert X_n\vert>\kappa(a))}]\\
&\leq \kappa(a)\PP(A)+\EE[\vert X_n\vert1_{A\cap (\vert X_n\vert>\kappa(a))}].
\end{align*}
Using now that $x>\kappa(a)$ implies $g(x)/a\geq x$, we have
\[
\EE[\vert X_n \vert 1_A]\leq \kappa(a)\PP(A)+a^{-1}\EE[g(\vert X_n\vert)]<\kappa(a)\PP(A)+a^{-1}K.
\]
Setting $a:=2K/\varepsilon$ yields $\EE[\vert X_n \vert 1_A]\leq \kappa(2K/\varepsilon)\PP(A)+\varepsilon/2$, so that if $\PP(A)\leq\mu(\varepsilon)$ one has $\EE[\vert X_n \vert 1_A]\leq \varepsilon$.
\end{proof}

The above lemma is a quantitative variant of the fundamental de la Vall\'ee-Poussin theorem \cite{Pou1915} (see again also \cite[Theorem 6.19]{Kle2020}), and the existence of such a supercoercive function in fact characterizes uniformly integrable sequences of random variables. In particular, an immediate consequence is the following result for higher moment conditions:

\begin{lemma}
\label{uniformint:lem:moment}
Let $(X_n)$ be a sequence of random variables such that $\sup_{n\in\NN}\EE[|X_n|^p]<K$ for some $K>0$ and $p>1$. Then $(X_n)$ is uniformly integrable and $\mu$ defined by
\[
\mu(\varepsilon):=\frac{\varepsilon}{2}\left(\frac{\varepsilon}{2K}\right)^{1/(p-1)}
\]
is a corresponding modulus in the sense of Definition \ref{uniformint:def}.
\end{lemma}

Lemmas \ref{uniformint:lem} and \ref{uniformint:lem:moment} and can be applied directly to $(\norm{Uy_n-y_n})$, and of course in the special case that $X_n:=Y$ become simpler results on absolute continuity that can be instantiated to produce moduli of continuity on $Y$, and thus a modulus of integrability for $(\norm{Uy_n-y_n})$. However, when reducing the problem to continuity properties of $\sum_{i=0}^\infty \norm{\xi_i}$ and $\sum_{i=0}^\infty \norm{\delta_i}$, as possible under the assumptions of this section as shown in Lemma \ref{lem:ui:to:errors}, we conjecture that stronger assumptions such as those of Lemma \ref{uniformint:lem:moment} are not even required for the quantitative result, and that for concrete instantiations of the error terms (via e.g.\ minibatching as discussed in Section \ref{sec:complexity}), moduli of integrability for the sums $\sum_{i=0}^\infty \norm{\xi_i}$ and $\sum_{i=0}^\infty \norm{\delta_i}$ can be calculated explicitly, exploiting the fact that we have concrete knowledge of the distribution of the errors. However, we do not give further details here.

\section{Fast rates of asymptotic regularity}\label{sec:fast}

In this section, we focus on particular instantiations of the parameters together with suitable growth conditions on the errors that allow for fast rates of asymptotic regularity for the above iteration(s). For that, we begin with some general results on deriving linear rates of convergence for sequences of real numbers satisfying a general recursive inequality and we subsequently extend this to sequences of random variables and utilize these general results then to in turn derive the fast rates. Throughout the section, we will be very explicit about the exact kind of assumptions (i.e.\ \eqref{Hyp}, or \eqref{asHyp}, or the existence of common fixed points as in the last part of the previous section) that are placed on the iterations in question.

\subsection{General results on linear rates}

We begin with the crucial result on deriving fast rates of convergence for Halpern-style iterations in nonlinear optimization. This result is closely modelled after a seminal lemma by Sabach and Shtern \cite{SS2017}, first utilised in the context of proof mining in \cite{CKL2023}. Here we formulate the idea behind the lemma in a slightly different style to fit the iterations considered in this paper, and in this way our presentation is closer to the explicit closed-form bounds in \cite{BC2024}. However, as the proof is still nothing more than a careful implementation of the arguments given in \cite{SS2017}, we defer it to the appendix.

\begin{lemma}[essentially \cite{SS2017}]\label{sabach:stern}
Suppose that $(s_n),(c_n)$ are sequences of nonnegative real numbers satisfying
\[
s_{n+1}\leq (1-a_n)s_n+c_n
\]
for all $n\in\mathbb{N}$ where $(a_n)\subseteq [0,1]$. Then for all $m,K\in\NN$ we have
\[
s_{K+m+1}\leq A_K^{K+m}s_K+\sum_{i=K}^{K+m}A_{i+1}^{K+m}c_i
\]
for $A_j^k:=\prod_{i=j}^k(1-a_i)$, with $A_j^k:=1$ for $j>k$. In the special case that $a_n:=\alpha_{n+1}$ and $c_n\leq (\alpha_n-\alpha_{n+1})L$ for some $(\alpha_n)\subseteq [0,1]$ and $L>0$ we have
\[
s_n\leq \tilde A_1^n s_0+L\sum_{i=1}^n(\alpha_{i-1}-\alpha_i)\tilde A_{i+1}^n
\]
for all $n\in\mathbb{N}$ where $\tilde A_j^k:=\prod_{i=j}^k(1-\alpha_i)$, with $\tilde A_j^k:=1$ for $j>k$. If we furthermore define $\alpha_n:=2/(n+2)$ and assume that $s_0\leq L$, then $s_n\leq 2L/(n+2)$ for all $n\in\NN$.
\end{lemma}

The following is an adaptation of the special case of the previous lemma concerning fast rates to sequences of random variables and as we will see in the following, it assumes a similarly important role for deriving linear rates of almost sure convergence.

\begin{lemma}\label{sabach:stern:as}
Suppose that $(X_n)$, $(C_n)$ are nonnegative stochastic processes satisfying
\[
X_{n+1}\leq (1-\alpha_{n+1})X_n+C_n
\]
almost surely for any $n\in\mathbb{N}$ where $\alpha_n:=2/(n+2)$ and where $\EE[C_n]\leq (\alpha_n-\alpha_{n+1})L$ almost surely for all $n\in\NN$ where $L\geq \EE[X_0]$. Then 
\[
\EE[X_n]\leq\frac{2L}{n+2}\text{ and }\PP\left(\exists i\geq n\left( X_i\geq\varepsilon\right)\right)\leq\frac{1}{\varepsilon}\frac{4L}{n+2}
\]
for all $n\in\mathbb{N}$.
\end{lemma}
\begin{proof}
From the fact that $X_{n+1}\leq (1-\alpha_{n+1})X_n+C_n$ holds almost surely, we immediately derive $\EE[X_{n+1}]\leq (1-\alpha_{n+1})\EE[X_n]+\EE[C_n]$ and Lemma \ref{sabach:stern} yields $\EE[X_n]\leq 2L/(n+2)$. Proceeding as in the proof of Lemma \ref{prob:to:as}, noting that we in particular have $X_{n+1}\leq X_n+C_n$ almost surely, we similarly derive
\[
\PP(\exists n\geq N(U_n\geq\varepsilon))\leq \frac{1}{\varepsilon}\left(\EE[X_N]+\sum_{i=N}^\infty\EE[C_i]\right)
\]
for $U_n:=X_n+\sum_{i=n}^\infty C_i$. In particular, we have $\EE[X_N]\leq 2L/(N+2)$ and
\[
\sum_{i=N}^\infty\EE[C_i]\leq L\sum_{i=N}^\infty (\alpha_{i}-\alpha_{i+1})=L\alpha_N=\frac{2L}{N+2}
\]
so that $\PP(\exists n\geq N(U_n\geq\varepsilon))\leq \frac{1}{\varepsilon}\frac{4L}{N+2}$. This gives 
\[
\PP\left(\exists n\geq N\left( X_n\geq\varepsilon\right)\right)\leq\PP\left(\exists n\geq N\left(U_n\geq\varepsilon\right)\right)\leq \frac{1}{\varepsilon}\frac{4L}{N+2}
\]
again as in Lemma \ref{prob:to:as}.
\end{proof}

\begin{remark}
Note that from the conclusions of Lemma \ref{sabach:stern:as}, it is rather immediate to give corresponding rates for $\EE[X_n]\to 0$ and $X_n\to 0$ almost surely, e.g.\ by setting $\Phi(\lambda,\varepsilon):=\ceil*{4L/\varepsilon\lambda}$ for the latter, but we prefer the above formulations in this section to make the constants very explicit.
\end{remark}

\subsection{Linear rates of asymptotic regularity}

We now begin by establishing linear rates of asymptotic regularity for the iterations $(x_n)$ and $(y_n)$ in the special case of parameters 
\[
\alpha_n=\frac{2}{n+2}\text{ and }\beta_n=\beta\in (0,1).\tag{Par}\label{Par}
\]

\begin{theorem}\label{thm:fastX}
Let $(x_n),(y_n)$ be the sequences generated by \eqref{sHM} for the parameters as in \eqref{Par}. Assume \eqref{Hyp} with constant $K_0$. Also, assume that $\EE[\norm{\xi_n}]\leq K_1/(n+2)^2$ and $\EE[\norm{\delta_n}]\leq K_2/(n+2)^2$. Then
\[
\EE[\norm{x_n-x_{n+1}}]\leq \frac{2L}{n+2}\text{ and }\PP\left(\exists i\geq n\left( \norm{x_i-x_{i+1}}\geq\varepsilon\right)\right)\leq \frac{1}{\varepsilon}\frac{4L}{n+2}
\]
for all $n\in\NN$ and $\varepsilon>0$, where $L=2K_0+2K_1+2K_2$ in both cases. 
\end{theorem}
\begin{proof}
As in the proof of Theorem \ref{res:x:ar}, we have $X_{n+1}\leq (1-\alpha_{n+1})X_n+C_n$ for $X_n:=\norm{x_n-x_{n+1}}$ and
\[
C_n:=\norm{\xi_{n+1}-\xi_n}+\norm{\delta_{n+1}-\delta_n}+(\alpha_n-\alpha_{n+1})(\norm{Tx_n-u}+\norm{\xi_n}).
\]
Also following the proof of Theorem \ref{res:x:ar} we have
\[
\EE[X_0]=\EE[\norm{x_{0}-x_{1}}]\leq 2K_0+K_1+K_2\leq L.
\]
So it remains to show that $\EE[C_n]\leq (\alpha_n-\alpha_{n+1})L$, and for this it suffices to show that
\[
\EE[\norm{\xi_{n+1}-\xi_n}]\leq (\alpha_n-\alpha_{n+1})\cdot 2K_1\text{ and }\EE[\norm{\delta_{n+1}-\delta_n}]\leq (\alpha_n-\alpha_{n+1})\cdot 2K_2.
\]
We conclude by observing that
\begin{align*}
\EE[\norm{\xi_{n+1}-\xi_n}]&\leq \EE[\norm{\xi_{n+1}}]+\EE[\norm{\xi_n}]=\frac{K_1}{(n+3)^2}+\frac{K_1}{(n+2)^2}\\
&\leq \frac{2K_1}{(n+2)^2}\leq \frac{4K_1}{(n+2)(n+3)}=(\alpha_n-\alpha_{n+1})\cdot 2K_1
\end{align*}
and similarly for $(\delta_n)$ and $K_2$. The rates then follow from Lemma \ref{sabach:stern:as}.
\end{proof}

\begin{remark}\label{rem:sumToAS}
Before moving to the other asymptotic regularity results, we just briefly note that the asymptotic condition $\EE[\norm{\xi_n}]\leq K_1/(n+2)^2$ naturally implies that $\sum_{n=0}^\infty\EE[\norm{\xi_n}]<\infty$ with a rather simple rate of convergence that can be easily calculated from the fact that
\begin{align*}
\sum_{n=N}^\infty\EE[\norm{\xi_n}]&\leq K_1\sum_{n=N}^\infty\frac{1}{(n+2)^2}\leq K_1\sum_{n=N}^\infty\frac{1}{(n+1)(n+2)}\\
&= K_1\sum_{n=N}^\infty\left(\frac{1}{n+1}-\frac{1}{n+2}\right)= \frac{K_1}{N+1}
\end{align*}
for $N\geq 1$. Similarly, this applies to $\delta_n$ and $K_2$. In particular, as highlighted before in Remark \ref{rem:sumExpAS}, we have
\[
\PP(\exists n\geq N(\norm{\xi_n}\geq\varepsilon))\leq\sum_{n=N}^\infty\PP(\norm{\xi_n}\geq\varepsilon)\leq\sum_{n=N}^\infty\frac{\EE[\norm{\xi_n}]}{\varepsilon}\leq\frac{1}{\varepsilon}\frac{K_1}{N+1}\leq\frac{1}{\varepsilon}\frac{2K_1}{N+2}.
\]
\end{remark}

Now, in the case of sequence $(y_n)$, the above then immediately implies the following:

\begin{theorem}\label{thm:fastY}
Let $(x_n),(y_n)$ be the sequences generated by \eqref{sHM} for the parameters as in \eqref{Par}. Assume \eqref{Hyp} with constant $K_0$. Also, assume that $\EE[\norm{\xi_n}]\leq K_1/(n+2)^2$ and $\EE[\norm{\delta_n}]\leq K_2/(n+2)^2$. Then
\[
\EE[\norm{y_n-y_{n+1}}]\leq \frac{2L}{n+2}
\]
for all $n\in\NN$ where $L$ can be given as an integer linear combination of $K_0$, $K_1$ and $K_2$. If we assume \eqref{asHyp} with $K_0$ and $Y$, then
\[
\PP\left(\exists i\geq n\left( \norm{y_i-y_{i+1}}\geq\varepsilon\right)\right)\leq \frac{1}{\varepsilon}\frac{4L}{n+2}
\]
for all $n\in\mathbb{N}$ and $\varepsilon>0$, with a suitable $L$ constructed similarly.
\end{theorem}

The proof is rather routine and hence deferred to the appendix.

\begin{remark}\label{lem:ORemark}
Reformulated, the above results in particular state that if $(x_n),(y_n)$ are the sequences generated by \eqref{sHM} for parameters as in \eqref{Par} under the assumption \eqref{Hyp} and $\EE(\norm{\xi_n})=O(1/n^2)$ as well as $\EE(\norm{\delta_n})=O(1/n^2)$, then $\EE[\norm{x_n-x_{n+1}}]=O(1/n)$ as well as $\EE[\norm{y_n-y_{n+1}}]=O(1/n)$.
\end{remark}

\subsection{Linear rates of asymptotic regularity relative to the mappings in special cases}

We now discuss fast rates for the special cases \eqref{sH} and \eqref{sKMT}. As they are also routine, all proofs in the present section are deferred to the appendix.

In the special case of the stochastic Halpern iteration, which we reobtain (as discussed before) by setting $U:=\mathrm{Id}$ as well as $\delta_n:=0$, we get the following fast rates:

\begin{theorem}\label{thm:fastHalpern}
Let $(x_n),(y_n)$ be the sequences generated by \eqref{sHM} for parameters as in \eqref{Par} and where $U:=\mathrm{Id}$ and $\delta_n:=0$. Assume \eqref{Hyp} with constant $K_0$. Also, assume that $\EE[\norm{\xi_n}]\leq K_1/(n+2)^2$. Then
\[
\EE[\norm{Tx_n-x_n}]\leq \frac{2L}{n+2}
\]
for all $n\in\NN$ where $L$ can be given as an integer linear combination of $K_0$ and $K_1$. If we assume \eqref{asHyp} with $K_0$ and $Y$, then
\[
\PP\left(\exists i\geq n\left( \norm{Tx_i-x_i}\geq\varepsilon\right)\right)\leq\frac{1}{\varepsilon}\frac{4L}{n+2}
\]
for all $n\in\mathbb{N}$ and $\varepsilon>0$, with a suitable $L$ constructed similarly.
\end{theorem}

Theorem \ref{thm:fastHalpern} is closely related to \cite[Theorem 3.3]{BC2024}, but with adjusted step-sizes that now provide \emph{exact} linear rates (without logarithmic factors).

In the special case of the stochastic Krasnoselskii-Mann iteration with Tikhonov regularization terms, which we re-obtain by setting $T:=\mathrm{Id}$ as well as $\xi_n:=0$, we get the following fast rates in the above special case:

\begin{theorem}\label{thm:fastTM}
Let $(x_n),(y_n)$ be the sequences generated by \eqref{sHM} for parameters 
as in \eqref{Par} and where $T:=\mathrm{Id}$ and $\xi_n:=0$. Assume \eqref{Hyp} with constant $K_0$. Also, assume that $\EE[\norm{\delta_n}]\leq K_2/(n+2)^2$. Lastly, let $B\geq 1/(1-\beta)$. Then
\[
\EE[\norm{Ux_{n}-x_{n}}]\leq \frac{2L}{n+2}
\]
for all $n\in\NN^*$ where $L$ can be constructed in terms of $K_0$, $K_2$ and $B$. If we assume \eqref{asHyp} with $K_0$ and $Y$, then
\[
\PP\left(\exists i\geq n\left( \norm{Ux_i-x_i}\geq\varepsilon\right)\right)\leq\frac{1}{\varepsilon}\frac{4L}{n+2}
\]
for all $n\in\mathbb{N}^*$ and $\varepsilon>0$, with a suitable $L$ constructed similarly.
\end{theorem}

\subsection{Fast rates of asymptotic regularity relative to the mappings in the general case}

In the context of the above assumptions on the scalar sequences and the errors, we can still get rather sensible complexity estimates in the general case where neither mapping necessarily trivializes. While we could express this again using a general modulus of uniform convexity $\eta$ for the underlying space, we here focus on the case where $\eta$ is of \emph{power type $p$} for $p\geq 2$, i.e.\ there exists a constant $C$ such that $\eta(\varepsilon)=C\varepsilon^p$. Crucially, this is the case for $L^p$ spaces for general $p\in (1,\infty)$, and for arbitrary inner product spaces:

\begin{lemma}[essentially \cite{Cla1936}, see also \cite{KL2012b,LT1979}]\label{lem:ippPowerType}
If $X$ is an inner product space, then $X$ is uniformly convex with a corresponding modulus $\eta(\varepsilon)=\varepsilon^2/8$. Further, if $X=L^p$ for $p>1$, then $X$ is uniformly convex with a corresponding modulus 
\[
\eta(\varepsilon)=\begin{cases}\frac{p-1}{8}\varepsilon^2&\text{if }1<p<2,\\\frac{1}{p2^p}\varepsilon^p&\text{if }2\leq p<\infty.\end{cases}
\]
\end{lemma}

This allows for the following results on the asymptotic regularity in the general case. We begin by instantiating Theorem \ref{geometry:as} on the asymptotic regularity of the sequence $(y_n)$ relative to $U$ almost surely and Theorem \ref{geometry:alt} for deriving the respective regularity result in expectation.

\begin{lemma}\label{lem:fastUYgeneral}
Let $X$ be uniformly convex with a modulus $\eta$ of power type $p$ with constant $C$. Let $(x_n),(y_n)$ be the sequences generated by \eqref{sHM} for parameters as in \eqref{Par}. Let $K$ and $Y$ be as in Lemma \ref{lem:stochastic:bounds}. Also, assume that $\EE[\norm{\xi_n}]\leq K_1/(n+2)^2$ and $\EE[\norm{\delta_n}]\leq K_2/(n+2)^2$. Lastly, let $\Lambda>0$ be such that $\Lambda\leq\beta\leq 1-\Lambda$. Then $\norm{Uy_n-y_n}\to 0$ almost surely with rate
\[
\Gamma(\lambda,\varepsilon):=\ceil*{\frac{(3K)^{p-1} L}{C\Lambda^2\varepsilon^p\lambda^p}}
\]
for a suitable $L$ arising as an integer linear combination of $K$, $K_1$ and $K_2$. Given a modulus $\mu$ of uniform integrability for $(\norm{Uy_n-y_n})$, we further get $\EE[\norm{Uy_n-y_n}]\to 0$ with rate
\[
\Gamma'(\varepsilon):=\ceil*{\frac{2^p(3K)^{p-1} L}{C\Lambda^2\varepsilon^p\mu(\varepsilon/4)^p}}
\]
where $L$ is as above.
\end{lemma}
\begin{proof}
First note that in the context of moduli $\eta$ of power type $p\geq 2$, we are actually in the setting of the previous Remarks \ref{rem:arUYOpt} and \ref{rem:arUYOptExp} where $\tilde\eta(\varepsilon)=C\varepsilon^{p-1}$. Then the rate for $\norm{Uy_n-y_n}\to 0$ almost surely follows by instantiating the rate given in Theorem \ref{geometry:as} with the following moduli: With the above $\tilde\eta$, we have $\widehat{\varepsilon}:=C\Lambda^2\varepsilon^{p}\lambda^{p-1}/(3K)^{p-1}$ and $K':=3K/\lambda$.
Using Theorem \ref{thm:fastX}, we have $\PP\left(\exists i\geq n\left( \norm{x_i-x_{i+1}}\geq\varepsilon\right)\right)\leq \frac{1}{\varepsilon}\frac{4L_0}{n+2}$ for all $n\in\NN$ and a suitable constant $L_0$ arising as an integer linear combination of $K$, $K_1$ and $K_2$. So we in particular have that $\Delta(\lambda,\varepsilon)=\ceil*{4L_0/\varepsilon\lambda}$ is a corresponding rate for $\norm{x_n-x_{n+1}}\to 0$ almost surely. As $\alpha=2/(n+2)$, we further have rather immediately that $\rho(\varepsilon)=\ceil*{2/\varepsilon}$ is a corresponding rate for $\alpha_n\to 0$. Lastly, using the assumptions on $\norm{\xi_n}$ and $\norm{\delta_n}$, note that as in Remark \ref{rem:sumToAS} we have $\sum_{n=N}^\infty\EE[\norm{\xi_n}]\leq K_1/(N+1)$ so that $\chi_1(\varepsilon)=\ceil*{K_1/\varepsilon}$ is a corresponding rate of convergence for $\sum_{n=0}^\infty \EE[\norm{\xi_n}]<\infty$. The rate of convergence $\chi_2(\varepsilon)=\ceil*{K_2/\varepsilon}$ for $\sum_{n=0}^\infty \EE[\norm{\delta_n}]<\infty$ follows similarly. Then instantiating Theorem \ref{geometry:as} under Remark \ref{rem:arUYOpt} gives us the rate
\[
\max\left\{\ceil*{\frac{144L_0(3K)^{p-1}}{C\Lambda^2\varepsilon^{p}\lambda^{p}}},\ceil*{\frac{24K(3K)^{p-1}}{C\Lambda^2\varepsilon^{p}\lambda^{p}}},\ceil*{\frac{36K_1(3K)^{p-1}}{C\Lambda^2\varepsilon^{p}\lambda^{p}}},\ceil*{\frac{36K_2(3K)^{p-1}}{C\Lambda^2\varepsilon^{p}\lambda^{p}}}\right\}\leq \frac{(3K)^{p-1}L}{C\Lambda^2\varepsilon^p\lambda^p}
\]
for $L:=144L_0$, noting that $K,K_1,K_2\leq L_0$ and so the first part follows. For the second part, we just apply Lemma \ref{as:ui:mean} directly.
\end{proof}

Using that, we can then employ the previous Theorems \ref{res:other:ar} and \ref{res:other:ar:as} to derive rates of asymptotic regularity also for the sequence $(x_n)$ relative to the mappings $U$ and $T$. For simplicity, we now focus on inner product spaces, i.e.\ where $p=2$ and $C=1/8$ by Lemma \ref{lem:ippPowerType}.

\begin{theorem}\label{thm:fastGeneral}
Let $X$ be an inner product space. Let $(x_n),(y_n)$ be the sequences generated by \eqref{sHM} for parameters as in \eqref{Par}. Let $K$ and $Y$ be as in Lemma \ref{lem:stochastic:bounds}. Also, assume that $\EE[\norm{\xi_n}]\leq K_1/(n+2)^2$ and $\EE[\norm{\delta_n}]\leq K_2/(n+2)^2$. Lastly, let $\Lambda>0$ be such that $\Lambda\leq\beta\leq 1-\Lambda$. Then $\norm{Ux_n-x_n}\to 0$ and $\norm{Tx_n-x_n}\to 0$ almost surely with rates
\[
\Phi_1(\lambda,\varepsilon):=\ceil*{\frac{24KL}{\Lambda^2\varepsilon^2\lambda^2}}\text{ and }\Phi_2(\lambda,\varepsilon):=\ceil*{\frac{72KL}{\Lambda^2\varepsilon^2\lambda^2}},
\]
respectively, where $L$ is as in Lemma \ref{lem:fastUYgeneral}. Given a modulus $\mu$ of uniform integrability for $(\norm{Uy_n-y_n})$, we further get $\EE[\norm{Ux_n-x_n}]\to 0$ and $\EE[\norm{Tx_n-x_n}]\to 0$ with respective rates
\[
\varphi_1(\varepsilon):=\ceil*{\frac{96KL}{\Lambda^2\varepsilon^2\mu(\varepsilon/4)^2}} \text{ and }\varphi_2(\varepsilon):=\ceil*{\frac{288KL}{\Lambda^2\varepsilon^2\mu(\varepsilon/4)^2}}.
\]
\end{theorem}
\begin{proof}
The rates follow immediately by instantiating Theorems \ref{res:other:ar} and \ref{res:other:ar:as} with the rates obtained from Theorem \ref{thm:fastX} and Lemma \ref{lem:fastUYgeneral}, noting in particular that the quadratic rates for $\norm{Uy_n-y_n}\to 0$ and $\EE[\norm{Uy_n-y_n}]\to 0$ dominate.
\end{proof}

\begin{remark}\label{rem:fastRateGen}
Compared to the previous nonasymptotic guarantees for the special cases of \eqref{sH} and \eqref{sKMT}, the fast rates given for the general schema \eqref{sHM} given above might look slightly complex at first sight. However, note that the rates from Theorem \ref{thm:fastGeneral} can be simply bounded as 
\[
\Phi_1(\lambda,\varepsilon),\Phi_2(\lambda,\varepsilon)\leq \frac{R}{\varepsilon^2\lambda^2}+1\in \mathcal{O}(\varepsilon^{-2}\lambda^{-2})
\]
and
\[
\varphi_1(\varepsilon),\varphi_2(\varepsilon)\leq \frac{R}{\varepsilon^2\mu(\varepsilon/4)^2}+1\in \mathcal{O}(\varepsilon^{-2}\mu(\varepsilon/4)^{-2}),
\]
where, in each case, $R$ is a suitably large constant and $\mu$ is still a modulus of uniform integrability for $(\norm{Uy_n-y_n})$.
\end{remark}

\section{An outlook onto applications}\label{sec:outlook}

Our stochastic Halpern-Mann scheme \eqref{sHM} and the subsequent convergence analysis have been presented in a completely abstract way. While will consider this level of abstraction to be a virtue, in that it potentially encompasses many different scenarios and shows the interrelations of the different quantitative ingredients in the final rates, we nevertheless conclude with an informal discussion on how various aspects of our analysis can be both extended to encompass alternative notions of complexity and interpreted in a concrete way. We however emphasise that a more detailed study of the various applications of both \eqref{sKMT} and \eqref{sHM} will be provided in a forthcoming paper, so that our comments remain at the level of an extended sketch.

\subsection{Oracle complexity and managing variance}
\label{sec:complexity}

Recall that the intuition behind our abstract stochastic scheme is that $\xi_n$ and $\delta_n$ arise by evaluating stochastic oracles $\tilde T$ and $\tilde U$ for nonexpansive mappings $T$ and $U$ respectively. Though our primary focus has been on establishing direct convergence rates for our schemes, in concrete applications it is the resulting \emph{oracle complexity} that may act as a more reliable estimate of the actual cost of running the algorithm. To address this at a level of generality in line with our overall approach, we propose an abstract characterisation of oracle complexity: We first introduce two sequences of natural numbers ($\#\xi_n$) and ($\#\delta_n$), with the intuition that $\#\xi_n$ denotes the number of calls made to the stochastic oracle $\tilde T$ in computing the approximation to $Tx_n$ (with output $Tx_n+\xi_n$), and $\#\delta_n$ the number of calls to $\tilde U$ in computing the approximation to $Uy_n$. With this intuition in mind, we say that $\psi:\NN\to\NN$ is a bound on the oracle growth of the general scheme (\ref{sHM}) if
\[
\forall N\in\NN\, \left(\#\xi_N+\sum_{n=0}^{N-1} (\#\xi_n+\#\delta_n)\leq \psi(N)\right),
\]
so that the total number of oracle calls needed to compute $x_N$ and $y_N$ is bounded by $\psi(N)$. Note that the additional term $\#\xi_N$ occurs as we also calculate $y_N$, which requires us to evaluate $Tx_N+\xi_N$ as per the schema \eqref{sHM}.

Then, for example, if 
\[
\text{$\EE[\norm{Tx_n-x_n}]\to 0$ with rate $\varphi$,}
\]
it follows that for any $\varepsilon>0$, we can compute an approximant $x_N$ such that
\[
\text{$\EE[\norm{Tx_N-x_N}]<\varepsilon$ using at most $(\psi\circ \varphi)(\varepsilon)$ queries to the stochastic oracles.}
\]
Naturally, $N=\varphi(\varepsilon)$ suffices here, and this $N$ further has the property that $\EE[\norm{Tx_n-x_n}]<\varepsilon$ for all $n\geq N$. Indeed, the bounding function $\psi$ can be used in a similar way to convert any of our convergence results to a corresponding characterisation of the overall oracle complexity in a completely general manner, which can then be appropriately instantiated in specific applications. This in particular applies to our almost sure convergence results, but here we focus solely on convergence in mean.

In practice, there are a number of techniques that would allow us to approximate the mappings $T,U$ while at the same time achieving the necessary variance bounds needed to guarantee convergence within our framework. Minibatching is an obvious example: Here we assume that we have access to $T$ and $U$ via a pair of stochastic oracles $\tilde{T},\tilde{U}:X\times \Theta\to X$, defined over some suitable measure space $\Theta$. These oracles then give rise to a concrete minibatched Halpern-Mann iteration via
\[
\begin{cases}
y_n:=(1-\alpha_n)\frac{1}{k_n}\sum_{j=1}^{k_n}\tilde{T}(x_n,\zeta_{n,j})+\alpha_nu,\\
x_{n+1}:=(1-\beta_n)\frac{1}{l_n}\sum_{j=1}^{l_n}\tilde{U}(y_n,\zeta'_{n,j})+\beta_ny_n,
\end{cases}\tag{sHM-mini}\label{sHMm}
\]
where $(k_n)$ and $(l_n)$ are sequences of batchsizes, and for each $n\in\NN$, $\zeta_{n,1},\ldots,\zeta_{n,k_n}$ are independent samples drawn from some distribution $D_n$ over $\Theta$, and $\zeta'_{n,1},\ldots,\zeta'_{n,l_n}$ independent samples drawn from some distribution $D'_n$. We can recognize this as an instance of \eqref{sHM} by defining the corresponding errors $\xi_n:=\frac{1}{k_n}\sum_{j=1}^{k_n}\tilde{T}(x_n,\zeta_{n,j})-Tx_n$, and analogously for $\delta_n$. In particular, the resulting number of oracle queries are in this case simply given by the bath sizes, i.e.\ we have $\#\xi_n:=k_n$ and $\#\delta_n:=l_n$. 
 
Now, suppose that we impose some standard assumptions on our oracle $\tilde{T}$:\footnote{These variance assumptions on the oracles are indeed widely used, featuring in a number of recent works on stochastic Halpern schemes (see in particular \cite{BC2024,Caietal2022}).}
\begin{enumerate}
\item $\tilde{T}$ is unbiased, i.e.\ $\EE[\tilde{T}(x_n,\zeta_{n,1})]=Tx_n$ for all $n\in\NN$,
\item $\tilde{T}$ has controlled variance w.r.t.\ $x_n$, i.e.\ $\EE[\lVert\tilde{T}(x_n,\zeta_{n,1})-Tx_n\rVert^2]\leq \sigma$ for all $n\in\NN$ and some constant $\sigma$.
\end{enumerate}
Then whenever $X$ is a separable Hilbert space it immediately follows (using Jensen's inequality and sample independence) that $\EE[\norm{\xi_n}]^2\leq \sigma^2/k_n$. Further, if $X=(\RR^d,\norm{\cdot})$ for some (not necessarily Euclidean) norm $\norm{\cdot}$, we have $\EE[\norm{\xi_n}]^2\leq c^2\sigma^2/k_n$ where $c>0$ is such that $\norm{x}\leq c\norm{x}_E$ for all $x\in \RR^d$ (writing $\norm{x}_E$ for the Euclidean norm on $\RR^d$). In that case, we then obtain
\[
\sum_{n=0}^\infty \EE[\norm{\xi_n}]\leq \sum_{n=0}^\infty \frac{c\sigma}{\sqrt{k_n}},
\]
for some $c>0$, and so the variance bounds on $(\norm{\xi_n})$ that we require for convergence of $\eqref{sHM}$ can be ensured by choosing the batchsizes $(k_n)$ appropriately. Assuming analogous properties of the oracle $\tilde{U}$, the corresponding variance bounds for $(\norm{\delta_n})$ are achieved similarly through appropriate choices for $(l_n)$. We now discuss one such choice in the following example for the instantiation \eqref{sH}:

\begin{example}\label{ex:minibatch}
Consider the special case of the stochastic Halpern iteration \eqref{sH}, obtained from \eqref{sHM} by setting $U:=\mathrm{Id}$ and $\delta_n:=0$. The corresponding minibatched scheme \eqref{sHMm} in that case to
\[
x_{n+1}:=(1-\alpha_n)\frac{1}{k_n}\sum_{j=1}^{k_n}\tilde{T}(x_n,\zeta_{n,j})+\alpha_nu,\tag{sH-mini}\label{sHm}
\]
as previously considered in \cite{BC2024}. If $X$ is finite dimensional, then we have $\EE[\norm{\xi_n}]\leq c\sigma/\sqrt{k_n}$ for some constant $c>0$ as outlined above. Choosing $k_n=(n+1)^4$ hence yields $\EE[\norm{\xi_n}]\leq K_1/(n+2)^2$ for $K_1:=4c\sigma$. Under the additional conditions of Theorem \ref{thm:fastHalpern} (recalling that \eqref{Hyp} is satisfied whenever we assume $\mathrm{Fix}T\neq \emptyset$) we would have
\[
\EE[\norm{Tx_n-x_n}]\leq \frac{2L}{n+2}
\]
for $L$ as defined in Theorem \ref{thm:fastHalpern}, corresponding to a rate of convergence $\varphi(\varepsilon):=\varepsilon/2L$. For this choice of $(k_n)$, a bound on the oracle growth is given by $\psi(N)=N^5$, and therefore we can compute an $x_N$ such that
\[
\text{$\EE[\norm{Tx_N-x_N}]<\varepsilon$ using at most $\left(\varepsilon/2L\right)^5$ queries to the stochastic oracles}.
\]
In this way we obtain a version of \cite[Corollary 3.5]{BC2024} for linear rates \emph{without} logarithmic factors, and this represents just an extremely simple case within our overall framework. Such complexity results extend to the instantiation \eqref{sKMT} and a corresponding minibatched variant as will be discussed in detail in a forthcoming paper.
\end{example}
We conjecture that other methods of controlling variance beyond minibatching could also be expressed by instantiating $\xi_n,\delta_n$ in a suitable way, which would come with alternative characterisations of oracle complexity. We anticipate that, in general, obtaining good bounds on the oracle complexity for instances of (\ref{sHM}) will involve a tension between variance control and oracle growth, which in the simple case of minibatching corresponds, as detailed above, to a payoff between
\begin{equation*}
\text{convergence speed of $\sigma/\sqrt{k_n}\to 0$ and growth of $\sum_{n=0}^N k_n$}.%\label{tension}\tag{$\ast$}
\end{equation*}
Because our explicit rates of convergence make completely precise how improved control on the variance (through e.g.\ faster convergence rates for $\sum_{n=0}^\infty \EE[\norm{\xi_n}]$ and $\sum_{n=0}^\infty\EE[\norm{\delta_n}]$) leads to improved rates of convergence for the algorithm itself, we anticipate that in any given model that allows us to explicitly describe $\psi$ in terms of $\EE[\norm{\xi_n}]$ and $\EE[\norm{\delta_n}]$, our abstract quantitative results would not only provide us with a general bound on the associated oracle complexity in terms of parameters representing variance control, but might even provide insights into specific choices of parameters that optimize oracle complexity.

\subsection{Applications of \eqref{sHM}}

Our new stochastic iteration scheme (\ref{sHM}) (along with the associated convergence results) immediately leads to new algorithms which solve problems that can be formulated in terms of fixed points of nonexpansive mappings. Two immediate examples here are monotone inclusion problems (where we could potentially generalise the results of \cite{Caietal2022} which are based on a stochastic Halpern scheme), and splitting methods, where this time our novel stochastic Krasnoselskii-Mann iteration with Tikhonov regularization terms is a direct generalisation of the nonstochastic method utilised for this purpose in \cite{BCM2019}, and so could be used to, for instance, compute zeroes of sums of maximally monotone operators that can only be accessed in a noisy way. Perhaps the most interesting application of our scheme lies in model-free reinforcement learning, where algorithms such as \emph{$Q$-learning} \cite{Wat1989} are naturally formulated as noisy methods for computing fixpoints. Only very recently has the special case \eqref{sH} of the scheme (\ref{sHM}) corresponding to Halpern's iteration been instantiated as a form of $Q$-learning \cite{BC2024}, and we in that way also perceive our general method as an expanded class of learning algorithms.

To illustrate this latter application in a little more detail, suppose that $(\mathcal{S},\mathcal{A},r,p)$ forms a Markov decision process (MDP) over some finite set of states $\mathcal{S}$ and actions $\mathcal{A}$, where if we choose action $a$ in state $s$, $r(s,a)$ represents an immediate reward and $p(s,a,t)$ the probability that we transition to state $t$ (see e.g.\ \cite{Put2014} for a standard reference). Instantiating not \eqref{sH} but \eqref{sKMT} in the style of $Q$-learning results in the method
\begin{equation}
Q_{n+1}(s,a):=(1-\beta_n)\left(U\left(\gamma_nQ_n\right)(s,a)+\delta_n(s,a)\right)+\beta_n \left(\gamma_nQ_n(s,a)\right) \label{KMTQ}\tag{KM-T-Q},
\end{equation}
where $U$ is some suitable nonexpansive operator on $\mathbb{R}^{\mathcal{S}\times\mathcal{A}}$ with respect to
\[
\norm{Q}_\infty:=\max_{s\in\mathcal{S}}\max_{a\in\mathcal{A}}Q(s,a)
\]
that captures the underlying Bellman equation $UQ=Q$, which in the case of discounted $Q$-learning would be given by
\[
UQ(s,a):=r(s,a)+c\sum_{t\in S}p(s,a,t)\max_{b\in A}Q(t,b)
\]
for some discount factor $c\in [0,1)$ (so that $U$ in this case is even a strict contraction). The method \eqref{KMTQ}, which can be perceived as an extension of ordinary $Q$-learning with Tikhonov regularization terms, is to the best of our knowledge already novel, even for the discounted case, though we consider it primarily of interest in the cases where $U$ is properly nonexpansive. Notably, we claim that under suitable assumptions it could be used to compute optimal policies for MDPs in the \emph{averaged reward} setting, as will be explored in a forthcoming work, just as Halpern's scheme has been recently utilised to this end in \cite{BC2024} (under a minibatch strategy). The key aspect of those variants in that context, already motivating \cite{BC2024}, is that they stay computationally effective over the averaged reward setting, as the iterations allow for fast asymptotic behavior even in the presence of general nonexpansive maps, and the convergence results presented here would allow us to produce similar guarantees, which could then be translated to sample complexity bounds by introducing oracle complexity terms as discussed in Section \ref{sec:complexity}. Concrete instantiations of the noise terms via minibatching are also possible in the context of reinforcement learning, in that they allow for sufficient variance control through suitable choices of batchsizes (see \cite{BC2024}), however the precise details are quite technical and we do not discuss them further here. In particular, this requires a more subtle version of variance control than that outlined in the previous section, where the associated oracle no longer has uniformly bounded variance in the sense of (2).

Ultimately, we envisage a broader use of our full framework in the context of reinforcement learning. Here our level of generality would allow us to consider $Q$-learning on a more abstract level, where we could take as a starting point the generalised model of \cite{LS1996} and explore both different iterative methods and different forms of variance reduction, using the results given here to formulate abstract convergence theorems for learning algorithms. In particular we could consider \emph{alternating} $Q$-learning algorithms based our main scheme (\ref{sHM}), for example
\begin{equation}
\begin{aligned}
Q'_n(s,a):=(1-\alpha_n)\left(TQ_n(s,a)+\xi_n(s,a)\right)+\alpha_n\hat{Q}(s,a)\\
Q_{n+1}(s,a):=(1-\beta_n)\left(UQ'_n(s,a)+\delta_n(s,a)\right)+\beta_nQ'_n(s,a)
\end{aligned}\label{sHMQ}\tag{sHM-Q}
\end{equation}
where $(\mathcal{S},\mathcal{A},r',p')$ is a second MDP over the same states and actions giving rise to the respective Bellman-type operator $T$. This iteration thus combines the traditional $Q$-learning procedure with a Halpern variant. This method bears a passing resemblance to double $Q$-learning \cite{Hass2010}, which incorporates a double estimator to reduce bias in ordinary $Q$-learning. Moreover, our scheme converges to a simultaneous fixpoint of two underlying nonexpansive operators, and thus might be relevant in situations where we are required to compute optimal policies concurrently across distinct environments. However, we leave an exploration of the potential merits of (\ref{sHMQ}) and similar algorithms to future work. In particular, a proper study of scenarios in which they could be exploited, together with an assessment of their performance against state-of-the-art reinforcement learning algorithms, would require substantial empirical work that is beyond the scope of the present paper.

\bibliographystyle{plain}
\bibliography{ref}

\section*{Appendix}

\begin{proof}[Proof of Theorem \ref{res:other:ar}]
Related to (a) -- (d), we can immediately establish the following inequalities:
\begin{align*}
\norm{x_n-y_n}&\leq \norm{x_{n+1}-x_n}+\norm{x_{n+1}-y_n}\\
&\leq \norm{x_{n+1}-x_n}+\norm{Uy_n-y_n}+\norm{\delta_n},\\
\norm{Ty_n-y_n}&\leq \norm{Ty_n-Tx_n}+\norm{Tx_n-y_n}\\
&\leq \norm{y_n-x_n}+\alpha_n\norm{Tx_n-u}+\norm{\xi_n},\\
\norm{Ux_n-x_n}&\leq \norm{Ux_n-Uy_n}+\norm{Uy_n-y_n}+\norm{y_n-x_n}\\
&\leq 2\norm{x_n-y_n}+\norm{Uy_n-y_n},\\
\norm{Tx_n-x_n}&\leq \norm{Tx_n-Ty_n}+\norm{Ty_n-y_n}+\norm{y_n-x_n}\\
&\leq 2\norm{x_n-y_n}+\norm{Ty_n-y_n}.
\end{align*}
By taking the expectation, the rates immediately follow.
\end{proof}

\begin{proof}[Proof of Theorem \ref{res:other:ar:as}]
The results follow immediately by the same inequalities established in the proof of Theorem \ref{res:other:ar} where in the case (b) one just needs the following additional consideration, giving a rate for $\alpha_n\norm{Tx_n-u}\to 0$ almost surely: Using Markov's inequality we have
\[
\PP(\exists n\left(\norm{Tx_n-u}\geq K_0/\lambda\right))\leq\PP(Y\geq K_0/\lambda)\leq \lambda
\]
for $Y$ as in \eqref{asHyp}. Now noting that if $\omega$ is such that $\norm{Tx_n(\omega)-u(\omega)}< K_0/\lambda$ for all $n\in\mathbb{N}$, then $\alpha_n\norm{Tx_n(\omega)-u(\omega)}<\varepsilon$ for any $n\geq \rho(\varepsilon\lambda/K_0)$, and therefore we have
\[
\PP(\exists n\geq \rho(\varepsilon\lambda/K_0)(\alpha_n\norm{Tx_n-u}\geq \varepsilon))\leq\PP(\exists n\left(\norm{Tx_n-u}\geq K_0/\lambda\right))\leq \lambda
\]
for any $\varepsilon,\lambda>0$ and this suffices to establish the claim in this case.
\end{proof}

\begin{proof}[Proof of Lemma \ref{lem:TMineq}]
Given $n\in\mathbb{N}$, we have
\begin{align*}
\norm{Uy_{n+1}-y_{n+1}}&=\norm{Uy_{n+1}-(1-\alpha_{n+1})x_{n+1}-\alpha_{n+1}u}\\
&\leq \norm{Uy_{n+1}-x_{n+1}}+\alpha_{n+1}\norm{Uy_{n+1}-u}\\
&\leq \norm{Uy_{n+1}-Uy_n}+\beta_n\norm{Uy_{n+1}-y_n}+\norm{\delta_n}+\alpha_{n+1}\norm{Uy_{n+1}-u}\\
&\leq 2\norm{y_{n+1}-y_n}+\beta_n\norm{Uy_{n+1}-y_{n+1}}+\norm{\delta_n}+\alpha_{n+1}\norm{Uy_{n+1}-u}
\end{align*}
pointwise everywhere.
\end{proof}

\begin{proof}[Proof of Theorem \ref{thm:UAsRegKMT}]
By the Lemma \ref{lem:TMineq} above, we have
\[
(1-\beta_n)\norm{Uy_{n+1}-y_{n+1}}\leq 2\norm{y_n-y_{n+1}}+\alpha_{n+1}\norm{Uy_{n+1}-u}+\norm{\delta_n}
\]
pointwise everywhere for any $n\in\mathbb{N}$. After taking expectations, we have
\[
\EE[\norm{Uy_{n+1}-y_{n+1}}]\leq \frac{1}{\Lambda}\left(2\EE[\norm{y_n-y_{n+1}}]+\alpha_{n+1}K_0+\EE[\norm{\delta_n}]\right)
\]
and from that the first rate immediately follows. The second rate follows rather similarly from the above relation: Using Markov's inequality, it holds that 
\[
\PP\left(\exists n \left(\norm{Uy_{n+1}-u}\geq\frac{K_0}{\lambda}\right)\right)\leq \PP\left(Y\geq\frac{K_0}{\lambda}\right)\leq\lambda
\]
for all $\lambda>0$. Let now $\lambda,\varepsilon>0$ be given. Take then $\omega$ such that $\norm{y_n(\omega)-y_{n+1}(\omega)}<\Lambda\varepsilon/4$ for all $n\geq \psi(\lambda/3,\Lambda\varepsilon/4)$ and $\norm{\delta_n(\omega)}<\Lambda\varepsilon/4$ for all $n\geq \phi(\lambda/3,\Lambda\varepsilon/4)$ as well as $\norm{Uy_{n+1}(\omega)-u(\omega)}\leq\frac{K_0}{\lambda}$ for all $n$. Then for $n\geq \zeta(\lambda,\varepsilon)-1$, it follows from the above inequality that $\norm{Uy_{n+1}(\omega)-y_{n+1}(\omega)}<\varepsilon$. This immediately yields $\PP\left(\exists n\geq \zeta(\lambda,\varepsilon)\left(\norm{Uy_{n}-y_{n}}\geq\varepsilon\right)\right)<\lambda$ which completes the proof.
\end{proof}

\begin{proof}[Proof of Lemma \ref{sabach:stern}]
The first inequality follows for all $m,K\in\NN$ as in Lemma \ref{lem:quantXuLem} immediately by induction. For the second part of the above lemma, we note that $A_j^k=\tilde A_{j+1}^{k+1}$ and thus from this first inequality (after setting $K=0$), we have
\[
s_{n+1}\leq A_0^n s_0+\sum_{i=0}^{n}A_{i+1}^{n}c_i=\tilde A_1^{n+1} s_0+\sum_{i=0}^{n}\tilde A_{i+2}^{n+1}(\alpha_i-\alpha_{i+1})L\leq \tilde A_1^{n+1}s_0+L\sum_{i=1}^{n+1}\tilde A_{i+1}^{n+1}(\alpha_{i-1}-\alpha_i)
\]
and also $s_0=\tilde A_{1}^0 s_0$ by definition. For the final part, we observe that
\[
\tilde A_i^n=\prod_{j=i}^n(1-\alpha_{j})=\prod_{j=i}^n\frac{j}{j+2}=\frac{i}{i+2}\cdot \frac{i+1}{i+3}\cdot\ldots\cdot\frac{n-1}{n+1}\cdot\frac{n}{n+2}=\frac{i(i+1)}{(n+1)(n+2)}
\]
for $i\leq n$ (noting that this also holds for $i=n$), and therefore
\begin{equation*}
\begin{aligned}
s_n&\leq \frac{2s_0}{(n+1)(n+2)}+L\sum_{i=1}^n\left(\frac{2}{i+1}-\frac{2}{i+2}\right)\frac{(i+1)(i+2)}{(n+1)(n+2)}\\
&=\frac{2}{(n+1)(n+2)}\left(s_0+L\sum_{i=1}^n\left(\frac{1}{i+1}-\frac{1}{i+2}\right)(i+1)(i+2)\right)\\
&\leq\frac{2L}{(n+1)(n+2)}\left(1+\sum_{i=1}^n1\right)=\frac{2L}{n+2}
\end{aligned}
\end{equation*}
which completes the proof.
\end{proof}

\begin{proof}[Proof of Theorem \ref{thm:fastY}]
Using Lemma \ref{res:recurrence:y-to-x-to-y}, \eqref{eqn:y-to-x}, we have
\[
\norm{y_{n+1}-y_n}\leq \norm{x_{n+1}-x_n}+\norm{\xi_{n+1}-\xi_n}+\alpha_n\left(\norm{Tx_n-u}+\norm{\xi_n}\right)
\]
for all $n\in\mathbb{N}$ pointwise everywhere. By taking expectations, we get
\[
\EE[\norm{y_{n+1}-y_n}]\leq \EE[\norm{x_{n+1}-x_n}]+\EE[\norm{\xi_{n+1}-\xi_n}]+\alpha_{n}(K_0+K_1).
\]
From Theorem \ref{thm:fastX}, we get $\EE[\norm{x_n-x_{n+1}}]\leq2L/(n+2)$ and similar as in the proof thereof, we have $\EE[\norm{\xi_{n+1}-\xi_n}]\leq 2K_1/(n+2)^2\leq 2K_1/(n+2)$. Combined with the definition of $\alpha_n$, we get the first claim for a suitable $L$ arising as an integer linear combination of $K_0$, $K_1$ and $K_2$. The second claim follows similarly, noting the above Remark \ref{rem:sumToAS} and the fact that, using Markov's inequality, we have
\begin{align*}
&\PP\left(\exists n\geq N\left(\alpha_n(\norm{Tx_n-u}+\norm{\xi_n})\geq \varepsilon\right)\right)\\
&\qquad\leq \PP\left(\exists n\geq N\left((\norm{Tx_n-u}+\norm{\xi_n})\geq\varepsilon/\alpha_N\right)\right)\\
&\qquad\leq \PP\left(\exists n\geq N\left(Y\geq\varepsilon/2\alpha_N\right)\right)+\PP\left(\exists n\geq N\left(\norm{\xi_n}\geq\varepsilon/2\alpha_N\right)\right)\\
&\qquad\leq \frac{1}{\varepsilon}\frac{4K_0}{N+2}+\frac{1}{\varepsilon}\frac{8K_1}{N+2},
\end{align*}
which, combined with the previous, rather immediately yields the result (which we therefore do not spell out any further).
\end{proof}

\begin{proof}[Proof of Theorem \ref{thm:fastHalpern}]
Using the inequalities listed in the proof of Theorem \ref{res:other:ar}, we obtain
\[
\norm{Tx_n-x_n}\leq 3\norm{x_{n+1}-x_n}+\alpha_n\norm{Tx_n-u}+\norm{\xi_n}
\]
for all $n\in\mathbb{N}$ pointwise everywhere, in this special case where $U=\mathrm{Id}$ and $\delta_n=0$. This immediately yields the above rates (using similar arguments as in Lemma \ref{thm:fastY} in the case of the almost sure convergence) using the previous Theorem \ref{thm:fastX} (noting that in this case $K_2=0$).
\end{proof}

\begin{proof}[Proof of Theorem \ref{thm:fastTM}]
Using Lemma \ref{lem:TMineq}, we have
\[
\norm{Uy_{n}-y_{n}}\leq B\left(2\norm{y_{n-1}-y_{n}}+\alpha_{n}\norm{Uy_n-u}+\norm{\delta_{n-1}}\right)
\]
for $n\geq 1$ pointwise everywhere. Using the preceding Theorem \ref{thm:fastY}, we immediately get that
\[
\EE[\norm{Uy_n-y_n}]\leq\frac{2L_0}{n+2}\text{ and }\PP\left(\exists i\geq n\left( \norm{Uy_i-y_i}\geq\varepsilon\right)\right)\leq \frac{1}{\varepsilon}\frac{4L_0}{n+2}
\]
for $n\geq 1$ and a suitable constant $L_0$ arising as an integer linear combination of $K_0$ and $K_2$. Using the inequalities from Theorem \ref{res:other:ar}, we then further have
\[
\norm{Ux_n-x_n}\leq 2\norm{x_{n+1}-x_n}+3\norm{Uy_n-y_n}+2\norm{\delta_n}
\]
for all $n\in\mathbb{N}$ pointwise everywhere and so, using the previous results as well as Theorem \ref{thm:fastX}, we get the desired rates.
\end{proof}

\end{document}